\documentclass[a4j, 12pt, draft]{amsart}
\usepackage{amsmath,amssymb}	
\usepackage{amsthm}
\usepackage{mathrsfs}	
\usepackage{bm}
\usepackage{comment}
\usepackage{color}
\usepackage{mathtools}

\mathtoolsset{showonlyrefs=true}

\allowdisplaybreaks[4]

\theoremstyle{plain} 
\newtheorem{theorem}{Theorem}
\newtheorem{lemma}{Lemma}
\newtheorem{corollary}{Corollary}
\newtheorem{conjecture}{Conjecture}

\newtheorem{proposition}{Proposition}
\newtheorem{problem}{Problem}
\newtheorem*{conjecture*}{Conjecture}
\newtheorem*{theorem*}{Theorem}
\newtheorem*{assumption*}{Assumption}

\theoremstyle{plain}

\theoremstyle{remark}
\newtheorem{remark}{Remark}

\theoremstyle{definition}

\newtheorem*{notations*}{Notations}
\newtheorem*{acknowledgment*}{Acknowledgments}

\numberwithin{equation}{section}

\newcommand\swapcommand[2]{%
\let\swaptemp#1
\let#1#2
\let#2\swaptemp
}

\let\sl\l
\renewcommand\l{%
	\leavevmode
  \ifmmode
    \left
  \else
    \sl
  \fi}

\newcommand\set[2]{%
\left\{ #1 \; \middle| \; #2 \right\}
}

\swapcommand{\SS}{\S}
\renewcommand{\S}{\mathscr{S}}

\newcommand{\CC}{\mathbb{C}}
\newcommand{\RR}{\mathbb{R}}

\newcommand{\ZZ}{\mathbb{Z}}

\newcommand{\e}{\varepsilon}
\newcommand{\s}{\sigma}

\newcommand{\normal}{\normalfont}

\newcommand{\Lam}{\Lambda}

\newcommand{\us}{\underset}

\newcommand{\meas}{\operatorname{meas}}
\newcommand{\Li}{\operatorname{Li}}
\newcommand{\sgn}{\operatorname{sgn}}
\newcommand{\qqquad}{\qquad \qquad \qquad}
\newcommand{\qqqquad}{\qquad \qquad \qquad \qquad}

\newcommand{\T}{\mathscr{T}}

\renewcommand{\a}{\alpha}
\renewcommand{\b}{\beta}
\renewcommand{\r}{\right}
\renewcommand{\d}{\displaystyle}
\renewcommand{\Re}{\operatorname{Re}}
\renewcommand{\Im}{\operatorname{Im}}
\renewcommand{\epsilon}{\varepsilon}

\newcommand{\todaye}{\the\year/\the\month/\the\day}

\allowdisplaybreaks[0]

\title[]
{On the logarithm of the Riemann zeta-function and its iterated integrals}
\author[S. INOUE]{Sh\={o}ta Inoue}

\address{Graduate School of Mathematics, Nagoya University,
Furocho, Chikusaku, Nagoya 464-8602, Japan}
\email{m16006w@math.nagoya-u.ac.jp}



\begin{document}

\maketitle

\begin{abstract}
The present paper gives some results for the logarithm of the Riemann zeta-function and its iterated integrals.
We obtain a certain explicit approximation formula for these functions. 
The formula has some applications, which are related with the value distribution of these functions 
and a relation between prime numbers and the distribution of zeros in short intervals.
\end{abstract}




\section{\textbf{Introduction and statement of the main theorem}}


In the present paper, we discuss some properties of the logarithm of the Riemann zeta-function $\zeta(s)$ and its iterated integrals.
We define the function $\eta_{m}(s)$ by
\begin{gather*}
\eta_{m}(\s+it)
= \int_{0}^{t}\eta_{m - 1}(\s+it')dt' + c_{m}(\s),
\end{gather*}
where
\begin{gather*}
\eta_{0}(\s+it) = \log{\zeta(\s+it)}, \\
c_{m}(\s)
= \frac{i^{m}}{(m-1)!}\int_{\s}^{\infty}(\a - \s)^{m-1}\log{\zeta(\a)}d\a.
\end{gather*}
Here, we decide the branch\footnote{Some people adapt a slightly different definition in this case.} 
of the logarithm of the Riemann zeta-function as follows.
When $t$ is not equal to zero and the ordinate of nontrivial zeros of $\zeta(s)$, 
then we choose the branch by the continuation with the initial condition
$\lim_{\s \rightarrow +\infty}\log{\zeta(\s + it)} = 0$.
If $t = 0$, then $\log{\zeta(\s)} = \lim_{\e \downarrow 0}\log{\zeta(\s + i\e)}$. 
If $t$ is the ordinate of a nontrivial zero $\rho = \b + i\gamma$ of the Riemann zeta-function, then 
$\log{\zeta(\s + i\gamma)} = \lim_{\e \downarrow 0}\log{\zeta(\s + i(\gamma - \sgn(\gamma)\e))}$.
We also mention that the integral of the definition of $\eta_{m}(\s + it)$ is defined by the improper Riemann integral, 
that is, it is defined by the following.
If there are zeros $\rho_{j} = \b_{j} + i\gamma_{j}$ $(j = 1, \ldots, k)$ of $\zeta(s)$ 
satisfying $\s \leq \b_{j}$, $0 < \gamma_{j} \leq t$, 
then the integral of the definition of $\eta_{m}(\s + it)$ means that 
\begin{align*}
\eta_{m}(\s + it)
= \lim_{\e \downarrow 0}\sum_{j=0}^{k}\int_{\gamma_{j} + \e}^{\gamma_{j+1} - \e}\eta_{m-1}(\s+it')dt',
\end{align*} 
where $\gamma_{0} = 0, \gamma_{k+1} = t$.

Under the above definition, the well known function $S_{m}(t)$ can be represented by using $\eta_{m}(s)$. 
Actually, the function $S_{m}(t)$ is defined by
\begin{align*}
S_{m}(t) = \pi^{-1}\Im(\eta_{m}(1/2 + it)), 
\end{align*}
and particularly, we may write $S_{0}(t)$ as $S(t)$.
The study for $S(t)$ is important since this function has information on the distribution of zeros of $\zeta(s)$.
This fact can be understood by the Riemann-von Mangoldt formula:
\begin{align}	\label{RVMF}
N(T) = \pi^{-1}\arg\Gamma(1/4 + iT/2) - T\log{\pi}/2\pi + S(T) + 1.
\end{align}
Here, the function $N(T)$ is the number of zeros $\rho = \b + i\gamma$ of $\zeta(s)$ with 
$0 < \b < 1$, $0 < \gamma < T$ counted with multiplicity, 
and the function $\Gamma$ is the gamma-function.
Therefore, the function $S_{m}(t)$ being an $m$-th iterated integral of $S(t)$ is also a remarkable object, 
and the study for $S_{m}(t)$ has been done by many mathematicians.
For example, Littlewood \cite{L1924} and Selberg \cite{SS} showed $S_{m}(t) \ll_{m} \log{t} / (\log{\log{t}})^{m+1}$
for nonnegative integer $m$ under the Riemann Hypothesis.
It is also known in an unpublished work by Ghosh and Goldston (see pp.334--335 in \cite{T}) 
that the Lindel\"of Hypothesis is equivalent to the estimate $S_{1}(t) = o(\log{t})$.
Further, if the estimate $S_{1}(t) = o(\log{t} / (\log{\log{t}})^{2})$ holds, 
then we can obtain the interesting estimate $S(t) = o(\log{t} / \log{\log{t}})$.
This fact can be immediately obtained by Lemma 5 in \cite{CCM2013}.
Moreover, Fujii \cite{Fu2002} showed that the Riemann Hypothesis is equivalent to the assertion that, 
for any integer $m \geq 3$, the estimate $S_{m}(t) = o(T^{m-2})$ holds. 
Hence, we are interested in properties of $S_{m}(t)$.
On the other hand, we could expect that the real part of the logarithm of the Riemann zeta-function also 
has the information of zeros of $\zeta(s)$. 
Actually, the behavior of $\log{\zeta(s)}$ on $s$ close to a zero $\rho$ becomes roughly like $\log(s - \rho)$
whose real part is singular around the zero $\rho$. 
From this observation, it would be expected that the real part of $\eta_{m}(s)$ also has important information of zeros, 
and to understand clearly this observation, we show a certain explicit approximation formula for $\eta_{m}(s)$ in this paper.
The formula can be also applied to the value distribution of $\log{\zeta(1/2 + it)}$ and $\eta_{m}(s)$.
 

Throughout this paper, we use the following notations.

\begin{notations*}
Let $s = \s + it$ be a complex number with $\s$, $t$ real numbers, 
and $\rho = \b + i\gamma$ be a nontrivial zero of $\zeta(s)$ with $\b$, $\gamma$ also real numbers.
Let $\Lam(n)$ be the von Mangoldt function.

Let $H \geq 1$ be a real parameter.
The function $f: \RR \rightarrow [0, +\infty)$ is mass one and supported on $[0, 1]$, and further $f$ is a
$C^{1}([0, 1])$-function, or for some $d \geq 2$ $f$ belongs to $C^{d-2}(\RR)$ and is a $C^{d}([0, 1])$-function.
For such $f$'s, we define the number $D(f)$, and functions $u_{f, H}$, $v_{f, H}$ by 
\begin{align}	\label{def_D}
D(f) = \max\{ d \in \ZZ_{\geq 1} \cup \{+\infty\} \mid \text{$f$ is a $C^{d}([0, 1])$-function} \},
\end{align}
$u_{f, H}(x) = Hf(H\log(x/e))/x$, and
\begin{gather}
\label{def_v}
v_{f, H}(y) = \int_{y}^{\infty}u_{f, H}(x)dx,
\end{gather}
respectively. Further, for each integer $m \geq 0$, the function $U_{m}$ is defined by
\begin{gather}
\label{def_Um}
U_{m}(z) = \frac{1}{m!}\int_{0}^{\infty}\frac{u_{f, H}(x)}{(\log{x})^{m}}E_{m+1}^{*}(z\log{x})dx 
\end{gather}
for $\Im(z) \not= 0$.
Here, $E_{m+1}^{*}(z) = E_{m+1}^{*}(x+iy)$ is the function of a little modified $m$-th exponential integral defined by
\begin{align*}
E_{m+1}^{*}(z)
:= \int_{x+iy}^{+\infty+iy}(w - (x+iy))^{m}\frac{e^{-w}}{w} dw
= \int_{z}^{\infty}(w - z)^{m}\frac{e^{-w}}{w} dw.
\end{align*}
When $\Im(z) = 0$, then $U_{m}(x) = \lim_{\e \uparrow 0}U_{m}(x + i\e)$.

Let $X \geq 3$ be a real parameter.
The function $Y_{m}(s, X)$ is defined by
\begin{align*}
Y_{m}(s, X) = \l\{
\begin{array}{cl}
\d{\sum_{|s - \rho| \leq 1/\log{X}}\log((s - \rho)\log{X})}													&	\text{$m = 0$,}\\[7mm]
\d{2\pi \sum_{k = 0}^{m-1}\frac{i^{m-1-k}}{(m-k)! k!}\us{\b > \s}{\sum_{0 < \gamma < t}}(\b - \s)^{m-k}(t - \gamma)^{k}}
																															&	\text{$m \geq 1$.}
\end{array}
\r.
\end{align*}
In this paper, we take the branch of $\log{z}$ by $-\pi \leq \arg(z) < \pi$.
Here, we may represent $Y_{m}(s, X)$ by $Y_{m}(s)$ in the case  
$m \geq 1$ since $Y_{m}(s, X)$ does not depend on $X$ in this case.
\end{notations*}

\begin{remark}	\label{rmk_f}
From the above definitions, the function $u_{f, H}$ is mass one and supported on $[e, e^{1+1/H}]$, 
and further $u_{f, H}$ is a $C^{1}([e, e^{1+1/H}])$-function, 
or $u_{f, H}$ belongs to $C^{d-2}(\RR_{>0})$ and is a $C^{d}([e, e^{1+1/H}])$-function for some integer $d \geq 2$.
We also note that 
$v_{f, H}$ is a nonnegative continuous function on $\RR_{>0}$ and satisfies $v_{f, H}(y) = 0$ for $y \geq e^{1+1/H}$ and 
$v_{f, H}(y) = 1$ for $0 < y \leq e$.
\end{remark}

\begin{remark}	\label{rmk_Y}
Note that some remarks for $Y_{m}(s, X)$. When $m = 0$, the real part of it is always non-positive.
When $m = 1$, the function $Y_{1}(s)$ has the following simple formula
\begin{align*}
Y_{1}(s) = 2\pi \us{\b > \s}{\sum_{0 < \gamma < t}}(\b - \s),
\end{align*}
and its value is always nonnegative and always zero for $\s \geq 1/2$ under the Riemann Hypothesis.
Next, we suppose $m \geq 2$. Then if the Riemann Hypothesis is true, $Y_{m}(s)$ is always zero for $\s \geq 1/2$.
On the other hand, if the Riemann Hypothesis is false, the value of $Y_{m}(s)$ becomes big in $\s$ close to $1/2$. 
Actually, there exists a nontrivial zero $\rho_{0} = \b_{0} + i\gamma_{0}$ with $\b_{0} > 1/2$, then we have
\begin{gather}	\label{EVY}
\begin{gathered}
\Re(Y_{m}(s)) \geq (\b_{0} - \s)t^{m-1} + O\l(t^{m-3}\log{t}\r),\\
\Im(Y_{m}(s)) \geq (\b_{0} - \s)t^{m-2} + O\l(t^{m-4}\log{t}\r)
\end{gathered}
\end{gather}
for a fixed $\s$ with $1/2 \leq \s < \b_{0}$.
\end{remark}

Now, we state the main theorem in this paper. 

\begin{theorem}	\label{Main_Prop}
Let $m$, $d$ be a nonnegative integers with $d \leq D(f)$, and $H$, $X$ real parameters with $H \geq 1$, $X \geq 3$.
Then, for any $\s \geq 1/2$, $t \geq 14$, we have
\begin{align*}
\eta_{m}(s)
= i^{m}\sum_{2 \leq n \leq X^{1+1/H}}\frac{\Lam(n)v_{f, H}\l( e^{\log{n}/\log{X}} \r)}{n^{s}(\log{n})^{m + 1}}
+Y_{m}(s, X) + R_{m}(s, X, H).
\end{align*}
Here the error term $R_{m}(s, X, H)$ satisfies the estimate
\begin{multline}	\label{ESRm}
R_{m}(s, X, H)
\ll_{f, d} \frac{X^{2(1-\s)} + X^{1-\s}}{t(\log{X})^{m+1}}
+ \frac{1}{(\log{X})^{m}}\sum_{|t - \gamma| \leq \frac{1}{\log{X}}}(X^{2(\b - \s)} + X^{\b - \s})\\
+ \frac{1}{(\log{X})^{m + 1}}\sum_{|t - \gamma| > \frac{1}{\log{X}}}\frac{X^{2(\b - \s)} + X^{\b - \s}}{|t - \gamma|}
\min_{0 \leq l \leq d}\l\{\l( \frac{H}{|t - \gamma|\log{X}} \r)^{l}\r\}.
\end{multline}
Moreover, if the Riemann Hypothesis is true, for $1 \leq H \leq t/2$, $3 \leq X \leq t$, we have
\begin{align}	\label{ESRm2}
R_{m}(s, X, H)
\ll_{f} X^{1/2 - \s}\frac{\log{t}}{(\log{X})^{m}}\l(\frac{1}{\log{\log{t}}} + \frac{\log(H+2)}{\log{X}}\r).
\end{align}
\end{theorem}

The important point of this theorem is that, by $Y_{m}(s, X)$, we can express explicitly 
the contribution of certain zeros which have big influence to $\eta_{m}(s)$.
Actually, from this theorem, we can take out the information of singularities coming from such zeros.
Some consequences from this fact will be described in the next section.

Note some remarks on this theorem. 
First, when $m=0$, and $H$ is large, for example $H = X$, 
this formula becomes an assertion close to the hybrid formula of Gonek, Hughes, and Keating \cite[Theorem 1]{GHK2007}.
In fact, this theorem is proved by calculating the contribution of nontrivial zeros which is based on Proposition \ref{BPI}, 
and the proposition in the case of $H = X$, $m=0$ becomes the almost same as their formula.
On the other hand, as we can see from Theorem \ref{Main_Prop}, it becomes difficult to obtain a good estimate 
for the contribution of nontrivial zeros and mean value estimates when $H$ is large. 
From this reason, we introduce the new parameter $H$ which can control the length of smoothing functions.
Although most of discussions and results in the following are obtained by this theorem in the case $H$ is small, 
the theorem in the case $H$ is large is also useful 
when we discuss a Dirichlet polynomial without smoothing functions like $\sum_{p \leq X}p^{-1/2-it}$.
Actually, we will mention an estimate of this Dirichlet polynomial under the Riemann Hypothesis in inequality \eqref{GSO} below.

\section{\textbf{Applications of the main theorem}}

In this section, we state some consequences of Theorem \ref{Main_Prop}.
The consequences are related with the following:
\begin{itemize}
\item[1. ] An equivalence between the order of magnitude of  $\eta_{m}(s)$ and the zero-free region of $\zeta(s)$, 
\item[2. ] A relation between the prime numbers and the distribution of zeros of $\zeta(s)$ under the Riemann Hypothesis,
\item[3. ] The value distribution of $\log{|\zeta(1/2+it)|}$,
\item[4. ] A mean value theorem involving $\eta_{m}(s)$,
\item[5. ] The value distribution of $\eta_{m}(1/2+it)$.
\end{itemize}
We will state the details of these results in the following five sections.

\subsection{\textbf{An equivalence between the magnitude of the order of $\eta_{m}(s)$ and the zero-free region of $\zeta(s)$}}
\mbox{}

To begin with, we state a consequence which gives an equivalent condition to the zero-free region of $\zeta(s)$.
The consequence is the following.

\begin{corollary}	\label{IFR}
Let $\s \geq 1/2$.
Then the following three statements {\normal (A)}, {\normal (B)}, {\normal (C)} are equivalent.
\begin{itemize}
\setlength{\leftskip}{-5mm}
\item[(A).] The Riemann zeta-function does not have zeros whose real part are greater than $\s$.
\item[(B).] For a fixed integer $m \geq 2$, the estimate
\begin{align*}
\Re\eta_{m}(\s + iT) = o\l( T^{m-1} \r)
\end{align*}
holds as $T \rightarrow + \infty$.
\item [(C).] For a fixed integer $m \geq 3$, the estimate
\begin{align*}
\Im\eta_{m}(\s + iT) = o\l( T^{m-2} \r)
\end{align*}
holds as $T \rightarrow + \infty$.
\end{itemize}
In particular, for a fixed integer $m \geq 2$, the Riemann Hypothesis is equivalent to that the estimate 
\begin{align*}
\eta_{m}(1/2 + iT) = o\l( T^{m-1} \r)
\end{align*}
holds as $T \rightarrow +\infty$.
\end{corollary}

This corollary is easily obtained from Theorem \ref{Main_Prop}.
Actually, we can show it by the following little discussion.

Applying Theorem \ref{Main_Prop} as $X = 3$, $H = 1$, for any positive integer $m$, we can obtain the formula
\begin{align*}
\eta_{m}(s) = Y_{m}(s) + O_{m}\l( \sum_{\rho}\frac{1}{1 + (t - \gamma)^2} \r).
\end{align*}
Now, by the well known estimate (cf. p.98 \cite{DM}) 
\begin{align}	\label{BLDZ}
\sum_{\rho}\frac{1}{1 + (t - \gamma)^2} \ll \log{t},
\end{align}
the above $O$-term is $\ll_{m} \log{t}$.
Hence, we obtain
\begin{align}	\label{UBFeY}
\eta_{m}(s) = Y_{m}(s) + O_{m}(\log{t}).
\end{align}
Thus, from estimates \eqref{EVY} and \eqref{UBFeY}, we obtain Corollary \ref{IFR}.

Fujii \cite{Fu2002} showed an equivalence for the Riemann Hypothesis and an estimate for $S_{m}(t)$.
He discussed only the behavior of the Riemann zeta-function on the critical line, and
this corollary means that his equivalence can be generalized to the critical strip naturally.
Moreover, Fujii's result is an equivalence for $S_{m}(t)$ in the case $m \geq 3$.
On the other hand, thanks to the consideration on the real part of iterated integrals of the logarithm 
of the Riemann zeta-function, we also have the same type of equivalence for $m = 2$. \\[-3mm]

\subsection{\textbf{A Dirichlet polynomial involving prime numbers and the distribution of zeros of $\zeta(s)$ in short intervals}}	\label{sec_fz}
\mbox{}

In this section, we state some consequences of Theorem \ref{Main_Prop} for a relationship between prime numbers and
the distribution of nontrivial zeros of $\zeta(s)$ in short intervals. 
These consequences are obtained from a principle of taking out 
the information of singularities coming from certain zeros by using Theorem \ref{Main_Prop}.

We define the weighted Dirichlet polynomial $P_{f}(s, X)$ by
\begin{align}	\label{def_P}
P_{f}(s, X) = \sum_{p \leq X^2}\frac{v_{f, 1}(e^{\log{p}/\log{X}})}{p^{s}}
\end{align}
for $X \geq 3$. Here, the sum runs over prime numbers. 
Moreover, the function $\tilde{N}(t, h)$ means the number of zeros 
$\rho = \b + i\gamma$ of $\zeta(s)$ with $|t - \gamma| \leq h$ counted with multiplicity.
Then we can obtain the following theorem.

\begin{theorem}	\label{RELZZ}
Assume the Riemann Hypothesis. 
Let $f$ be a nonnegative mass one $C^{1}([0, 1])$-function supported in $[0, 1]$. 
Then, for $t \geq 14$, $\log{t} \leq X \leq t$, we have
\begin{multline}	\label{Lam_sum_ub}
P_{f}(1/2+it, X)
= \log{\l(\frac{\log{\log{t}}}{\log{X}}\r)} \times \tilde{N}\l(t, \frac{1}{\log{X}}\r) +\\
+ \sum_{\frac{1}{\log{X}} < |t - \gamma| \leq \frac{1}{\log{\log{t}}}}\log{\l( |t - \gamma|\log{\log{t}} \r)}
+ O_{f}\l( \frac{\log{t}}{\log{\log{t}}} \r).
\end{multline}
In particular, we have
\begin{gather}	\label{Lam_sum_ub_R_2}
\max_{3 \leq X \leq t}\Re\l(P_{f}(1/2+it, X) \r)
\ll_{f} \frac{\log{t}}{\log{\log{t}}},\\[3mm]
\label{Lam_sum_lb_R_2}
\max_{3 \leq X \leq t}\Re\l(-P_{f}(1/2+it, X) \r)
\ll_{f} \log{t},
\end{gather}
and
\begin{align}	\label{Lam_sum_ub_I_2}
\max_{3 \leq X \leq t}\l|\Im\l(P_{f}(1/2+it, X) \r)\r|
\ll_{f} \frac{\log{t}}{\log{\log{t}}}.
\end{align}
\end{theorem}

Here we focus on estimates \eqref{Lam_sum_ub_R_2}, \eqref{Lam_sum_ub_I_2}.
From these estimates, we would expect that it is possible to improve estimate \eqref{Lam_sum_lb_R_2} at $\log{t}/\log{\log{t}}$.
This expectation is coming from the following discussion.
By the randomness of the prime numbers, it is probably true that the numbers $\{t\log{p_1}\}, \dots, \{t\log{p_n}\}$ 
are randomly distributed on $[0, 1)$ for $t \geq 1$. Here, $\{x\}$ means the fractional part of $x$.
Hence, the author believes that there is not a big difference among the bounds of the real and imaginary parts of a weighted Dirichlet polynomial 
like $P_{f}(s, X)$ and their positive and negative parts.
From this observation, the author suggests the following conjecture.

\begin{conjecture}	\label{RC}
Let $\s$ be a real number, and $f$ be a nonnegative mass one $C^{1}([0, 1])$-function supported in $[0, 1]$.
For sufficiently large $T > 0$,
\begin{gather*}
\max_{14 \leq t \leq T}\max_{3 \leq X \leq t}\Re(P_{f}(\s+it, X)) \asymp \max_{14 \leq t \leq T}\max_{3 \leq X \leq t}\Re(-P_{f}(\s+it, X)), \\
\max_{14 \leq t \leq T}\max_{3 \leq X \leq t}\Re(P_{f}(\s+it, X)) \asymp \max_{14 \leq t \leq T}\max_{3 \leq X \leq t}\Im(P_{f}(\s+it, X)), 
\end{gather*}
and 
\begin{gather*}
\max_{14 \leq t \leq T}\max_{3 \leq X \leq t}\Im(P_{f}(\s+it, X)) \asymp \max_{14 \leq t \leq T}\max_{3 \leq X \leq t}\Im(-P_{f}(\s+it, X)).
\end{gather*}
\end{conjecture}

If this conjecture and the Riemann Hypothesis are true, for every certain $f$, we obtain
\begin{align}	\label{WFGKC}
\max_{3\leq X \leq t}|P_{f}(1/2+it, X)|
\ll \frac{\log{t}}{\log{\log{t}}}
\end{align}
from estimates \eqref{Lam_sum_ub_R_2}, \eqref{Lam_sum_ub_I_2}.

Estimate \eqref{WFGKC} can be applied to the distribution of the ordinate of zeros of $\zeta(s)$.
If estimate \eqref{WFGKC} and the Riemann Hypothesis are true, 
by using formula \eqref{Lam_sum_ub} as $X = (\log{t})^{D}$, 
we can obtain the following interesting estimate
\begin{align*}
\tilde{N}\l(t, \frac{1}{D\log{\log{t}}}\r)
\ll \frac{\log{t}}{\log{D} \log{\log{t}}}
\end{align*}
for any $2 \leq D \leq \log{t}/\log{\log{t}}$.
In particular, on the same condition, 
we can improve the estimate of the multiplicity of zeros of the Riemann zeta-function like the following
\begin{align*}
m(\rho) \ll \frac{\log|\gamma|}{(\log{\log{|\gamma|}})^2},
\end{align*}
where $m(\rho)$ means the multiplicity of a zero $\rho = \frac{1}{2} + i\gamma$.
This upper bound is sharp because the following inequality (see Corollary 1 in \cite{GG2007})
\begin{align*}
m(\rho) \leq \l( \frac{1}{2} + o(1) \r)\frac{\log{|\gamma|}}{\log{\log{|\gamma|}}}
\end{align*}
is the best known upper bound under the Riemann Hypothesis at present.
From this observation, the author suggests Conjecture \ref{RC} as an important open problem.

Furthermore, we will find a deeper fact from the same method as the above discussion.
We consider the following estimate
\begin{align}	\label{GSO}
\max_{3 \leq X \leq Y(t)}\l| \sum_{p \leq X}\frac{1}{p^{1/2 + it}} \r|
\leq M(t),
\end{align}
where $Y(t)$, $M(t)$ are some monotonically increasing functions 
with $3 \leq Y(t) \leq t$, $M(t) \ll \sqrt{Y(t)} / \log{Y(t)}$.
Note that an estimate of Dirichlet polynomial without a mollifier is useful 
because by partial summation and assuming estimate \eqref{GSO}, for any certain $f$, we have
$
P_{f}(1/2+it) \ll M(t)
$
for $3 \leq X \leq \sqrt{Y(t)}$.
This fact plays an important role in the following discussion in this section.

From the discussion in \cite[Section 2.2]{FGH2007}, we may expect that estimate \eqref{GSO} is true with 
$Y(t) = t$, $M(t) \asymp \d{\sqrt{\log{t}\log{\log{t}}}}$.
Here, we can obtain some bounds of $Y(t)$ and $M(t)$ under the Riemann Hypothesis.
Assuming the Riemann Hypothesis, by using estimate \eqref{ESRm2} as $H = X$, 
we can show that estimate \eqref{GSO} is true when $Y(t) = t$, $M(t) = \log{t}$.
Moreover, we can also show the inequality $M(t) \gg \sqrt{\log{t} \log{\log{\log{t}}} / \log{\log{t}}}$ 
when the inequality $Y(t) \geq \exp\l( L\sqrt{\log{t}\log{\log{t}} / \log{\log{\log{t}}}} \r)$ 
holds with $L$ sufficiently large constant.
This fact can be shown, for example, by the work of Bondarenko and Seip \cite[Theorem 2]{BS2018} 
and Selberg's formula \cite[Theorem 1]{SCR}.

Now, if estimate \eqref{GSO} and the Riemann Hypothesis are true, 
then we can obtain the following theorem.


\begin{theorem}	\label{VSIZD}
Assume the Riemann Hypothesis and estimate \eqref{GSO}. 
Let $\psi(t)$ be a function with $3 \leq \psi(t) \leq \sqrt{Y(t)}$.
Let $f$ be a nonnegative mass one $C^{1}([0, 1])$-function supported on $[0, 1]$. 
Then, for $t \geq 14$, $\psi(t) \leq X \leq t$, we have
\begin{multline*}
P_{f}(1/2+it, X) 
= \log\l( \frac{\log{\psi(t)}}{\log{X}} \r) \times \tilde{N}\l( t, \frac{1}{\log{X}} \r) +\\
+ \sum_{\frac{1}{\log{X}} < |t - \gamma| \leq \frac{1}{\log{\psi(t)}}}\log\l(|t - \gamma|\log{\psi(t)} \r)
+ O_{f}\l(M(t) + \frac{\log{t}}{\log{\psi(t)}} + \log{\log{X}}\r).
\end{multline*}
In particular, if the Riemann Hypothesis and estimate \eqref{GSO} 
with $Y(t) = t$, $M(t) \asymp \d{\sqrt{\log{t}\log{\log{t}}}}$ are true, 
then by taking $\psi(t) = \exp\l( \sqrt{\frac{\log{t}}{\log{\log{t}}}} \r)$, 
$X = \exp\l( D\sqrt{\frac{\log{t}}{\log{\log{t}}}} \r)$, we have
\begin{align}	\label{CEMZ}
\tilde{N}\l( t, \frac{\sqrt{\log{\log{t}}}}{D\sqrt{\log{t}}} \r) 
\ll \frac{\sqrt{\log{t}\log{\log{t}}}}{\log{D}}
\end{align}
for $3 \leq D \leq \frac{1}{2}\sqrt{\log{t} \log{\log{t}}}$.
\end{theorem}

By estimate \eqref{CEMZ}, assuming the Riemann Hypothesis 
and estimate \eqref{GSO} with $Y(t) = t$, $M(t) \asymp \sqrt{\log{t} \log{\log{t}}}$, we have
\begin{align}	\label{CEMZ2}
m(\rho) \ll \sqrt{\frac{\log{|\gamma|}}{\log{\log{|\gamma|}}}}.
\end{align}

Here, we should mention that, under the same condition, the estimate $m(\rho) \ll \d{\sqrt{\log{|\gamma|}\log{\log{|\gamma|}}}}$ 
immediately follows from Selberg's formula \cite[Theorem 1]{SS} and the Riemann-von Mangldt formula \eqref{RVMF}, 
and inequality \eqref{CEMZ2} is an improvement of this estimate.
Hence, from this observation, we may expect that there is an interesting relationship between 
the behavior of $\sum_{p \leq X}p^{-1/2-it}$ and the distribution of zeros of the Riemann zeta-function.

%
%
%
%

\subsection{\textbf{On the value distribution of $\log{|\zeta(1/2+it)|}$}}	\label{VDZ}
\mbox{}

In this section, we consider the value distribution of the Riemann zeta-function.
Now, we define the set $\S(T, V)$ by 
\begin{align}	\label{def_S}
\S(T, V) = 
\l\{ t \in [T, 2T] \mid \log{|\zeta(1/2+it)|} > V \r\}.
\end{align}
Here, we give a result on the value distribution of $\log{|\zeta(1/2+it)|}$.
There are interesting studies on this theme by Soundararajan \cite{SE2008}, \cite{SM2009}.
He showed a lower bound and an upper bound of the Lebesgue measure of $\S(T, V)$, 
and his result for the upper bound is under the Riemann Hypothesis.
In \cite{SM2009}, he mentioned the question that, in how large range of $V$, the following estimate
\begin{align}	\label{SCMU}
\frac{1}{T}\meas(\S\l(T, V\r))
\ll \frac{\sqrt{\log{\log{T}}}}{V}\exp\l(-\frac{V^2}{\log{\log{T}}}\r)
\end{align}
holds. Here, the symbol $\meas(\cdot)$ stands for the Lebesgue measure.
This problem is important because there are some interesting consequences such as 
the mean value estimate and the Lindel\"of Hypothesis.
Actually, if estimate \eqref{SCMU} holds for any large range of $V$, we can obtain the conjectural estimates
\begin{gather*}
\max_{t \in [T, 2T]}\log{|\zeta(1/2+it)|} \ll \sqrt{\log{T}\log{\log{T}}},\\
\int_{T}^{2T}|\zeta(1/2+it)|^{2k}dt \ll T(\log{T})^{k^2}.
\end{gather*}
Here, we should mention Jutila's work \cite{Ju1983}.
He showed unconditionally that the estimate
\begin{align*}
\frac{1}{T}\meas(\S(T, V))
\ll \exp\l( -\frac{V^2}{\log{\log{T}}}\l( 1 + O\l( \frac{V}{\log{\log{T}}} \r) \r) \r)
\end{align*}
holds for $0 \leq V \leq \log{\log{T}}$.
In particular, as an immediate consequence of this estimate, we have
\begin{align}	\label{JU}
\frac{1}{T}\meas(\S(T, V))
\ll \exp\l( -\frac{V^2}{\log{\log{T}}} \r)
\end{align}
for $0 \leq V \ll (\log{\log{T}})^{2/3}$.
This estimate does not slightly reach to estimate \eqref{SCMU}.
On the other hand, this estimate was improved by Radziwi\sl\sl\mbox{} \cite{Ra2011} in the shorter range 
$V = o\l( (\log{\log{T}})^{3/5 - \e} \r)$.
In fact, he showed that the following conjecture is true for $V = o\l( (\log{\log{T}})^{1/10-\e} \r)$.

\begin{conjecture*}[Radziwi\sl\sl, \cite{Ra2011}]
For $V = o\l(\sqrt{\log{\log{T}}}\r)$, as $T \rightarrow +\infty$
\begin{align*}
\frac{1}{T}\meas\l(\S\l(T, V \sqrt{\frac{1}{2}\log{\log{T}}}\r)\r)
\sim \int_{V}^{\infty}e^{-u^2 / 2}\frac{du}{\sqrt{2\pi}}.
\end{align*}
\end{conjecture*}
Hence, by his study, estimate \eqref{SCMU} have been proved for 
$\sqrt{\log{\log{T}}} \ll V = o\l( (\log{\log{T}})^{3/5 - \e} \r)$.
In this paper, we will extend unconditionally this range for $V$ to $\sqrt{\log{\log{T}}} \ll V \ll (\log{\log{T}})^{2/3}$.
Moreover, we will also show that the upper bound of Radziwi\sl\sl's conjecture is true for $V = o\l( (\log{\log{T}})^{1/6} \r)$.

\begin{theorem}	\label{LVEJ}
For $1 \ll V \ll (\log{\log{T}})^{1/6}$, we have
\begin{multline}	\label{UBS}
\frac{1}{T}\meas\l(\S\l(T, V \sqrt{\frac{1}{2}\log{\log{T}}}\r)\r)\\
\leq (1 + o(1))\int_{V}^{\infty}e^{-u^2/2}\frac{du}{\sqrt{2\pi}}
+ O\l(\frac{V}{(\log{\log{T}})^{1/3}}\exp\l(-\frac{V^2}{2}\r) \r)
\end{multline}
as $T \rightarrow +\infty$. 
In particular, for $1 \ll V = o\l((\log{\log{T}})^{1/6}\r)$, 
we have
\begin{align*}
\frac{1}{T}\meas\l(\S\l(T, V \sqrt{\frac{1}{2}\log{\log{T}}}\r)\r)
\leq (1 + o(1))\int_{V}^{\infty}e^{-u^2/2}\frac{du}{\sqrt{2\pi}}
\end{align*}
as $T \rightarrow +\infty$, and for any large $T$, we have
\begin{align}	\label{IJR}
\frac{1}{T}\meas(\S\l(T, V\r))
\ll \frac{\sqrt{\log{\log{T}}}}{V}\exp\l(-\frac{V^2}{\log{\log{T}}}\r)
\end{align}
for $\sqrt{\log{\log{T}}} \ll V \ll (\log{\log{T}})^{2/3}$.
\end{theorem}

Estimate \eqref{IJR} is an improvement of estimate \eqref{JU}, 
and it is expected from Radziwi\sl\sl's conjecture that the estimate is best possible.

This theorem will be shown by using a method of Selberg-Tsang \cite{KTDT} and Radziwi\sl\sl's method \cite{Ra2011}.
On the other hand, it would be difficult to prove Theorem \ref{LVEJ} by using their method only.
Actually, the author could not derive this theorem by a method using Lemma 5.4 in \cite{KTDT} 
which plays an important role in their method.
The reason why the author could not derive this theorem by such a method is that 
the contribution of zeros close to $s$ cannot be well managed.
On the other hand, we can ignore the contribution of such zeros by using Theorem \ref{Main_Prop} 
while considering the upper bound of $\meas\S(T, V)$. 
In fact, the important point in the proof of Theorem \ref{LVEJ} is that the real part of $Y_{0}(s, X)$ is always non-positive.

\subsection{\textbf{A mean value theorem involving $\eta_{m}(s)$}}
\mbox{}

In this section, we state a certain mean value theorem.
There are some interesting applications of the theorem to the value distribution of $\eta_{m}(s)$.

\begin{theorem}	\label{QvM}
Let $m$ be a positive integer.
Let $k$ be a positive integer. Let $T$ be large, and $X \geq 3$ with $X \leq T^{\frac{1}{135k}}$.
Then, for $\s \geq 1/2$, we have
\begin{multline*}
\frac{1}{T}\int_{14}^{T}\bigg| \eta_{m}(\s + it)
- i^{m}\sum_{2 \leq n \leq X}\frac{\Lam(n)}{n^{\s+it}(\log{n})^{m+1}} - Y_{m}(\s+it) \bigg|^{2k} dt\\
\ll 2^{k}k!\l( \frac{2m+1}{2m} + \frac{C}{\log{X}} \r)^{k} \frac{X^{k(1 - 2\s)}}{(\log{X})^{2km}}
+ C^{k} k^{2k(m+1)}\frac{T^{\frac{1-2\s}{135}}}{(\log{T})^{2km}}.
\end{multline*}
Here, the above $C$ is an absolute positive constant.
\end{theorem}

This theorem will give an answer for the question of how much of the function 
$\eta_{m}(s)$ can be approximated by the corresponding Dirichlet polynomial.
Such a study is often useful. 
For example, Radziwi\sl\sl \mbox{} \cite{Ra2011} proved a large deviation theorem for Selberg's limit theorem, 
and he used Corollary in \cite[p.60]{KTDT} to prove his result.
The corollary is related with the approximation of $\log{\zeta(s)}$ by a certain Dirichlet polynomial, 
and we can regard that Theorem \ref{QvM} corresponds to the corollary.
Hence, it is expected to be able to show a limit theorem for 
$\eta_{m}(s)$, which is similar to Selberg's limit theorem or the Bohr-Jessen limit theorem, and also its large deviation.
On the other hand, by using this theorem, we will show some results for the value distribution of $\eta_{m}(s)$ in the following.
Endo and the author showed the following theorem by using Theorem \ref{QvM}.

\begin{theorem*}[Endo and Inoue \cite{EIP} in preparation]
Let $1/2 \leq \s < 1$. 
If the number of zeros $\rho = \b + i\gamma$ with $\b > \s$ is finite, then the set 
\begin{align*}
\l\{\int_{0}^{t}\log\zeta(\s + it')dt' \; \middle| \; t \in [0, \infty) \r\}
\end{align*}
is dense in the complex plane.
Moreover, for each integer $m \geq 2$, the following statements are equivalent.
\begin{itemize}
\item[(I).] The Riemann zeta-function does not have zeros whose real part are greater than $\s$.
\item[(II).] The set
$
\l\{ \eta_{m}(\s + it) \mid t \in [0, \infty) \r\}
$
is dense in the complex plane.
\end{itemize}
\end{theorem*}
In particular, it follows from this theorem that the Riemann Hypothesis implies that the set
\begin{align*}
\set{\int_{0}^{t}\log{\zeta(1/2 + it')}dt'}{t \in [0, \infty)}
\end{align*}
is dense in the complex plane.
The motivation of this study is to give a new information for the following interesting open problem.

\begin{problem}	\label{DP}
Is the set $\l\{ \log{\zeta(1/2 + it)} \mid t \in \RR \r\}$ dense in the complex plane?
\end{problem}

There are some works for this problem such as \cite{GS2014}, \cite{KN2012}.
As we can see from those studies, the resolution of this problem is difficult at present.
On the other hand, we already know the following results as previous works for this problem.

\begin{theorem*}[Bohr and Courant in 1914 \cite{BC1914}]
For fixed $\frac{1}{2} < \s \leq 1$, the set $\l\{ \zeta(\s + it) \mid t \in \RR \r\}$ is dense in the complex plane.
\end{theorem*}

\begin{theorem*}[Bohr in 1916 \cite{B1916}]
For fixed $\frac{1}{2} < \s \leq 1$, the set $\l\{ \log{\zeta(\s + it)} \mid t \in \RR \r\}$ is dense in the complex plane.
\end{theorem*}

Note that the latter theorem is an improvement of former one 
since the former one is an immediate consequence from the latter theorem.
These results are interesting, and there are many developments such as
the Bohr-Jessen limit theorem \cite{BJ1930} and Voronin's universality theorem \cite{V1975}.
On the other hand, the value distribution of $\zeta(s)$ on the critical line is more difficult, 
and the resolution of Problem \ref{DP} is also difficult at present even under the Riemann Hypothesis.
From this viewpoint, the above theorem of Endo and the author is interesting,
and hencce Theorem \ref{QvM} is also important as a step to understand Problem \ref{DP}.


\subsection{\textbf{On the value distribution of $\eta_{m}(1/2+it)$}} \label{sec_vdeta}
\mbox{} 


In this section, we consider the value distribution of $\eta_{m}(1/2+it)$.
There are many studies on the value distribution of the Riemann zeta-function and other $L$-functions.

We discuss a measure for the difference between $\eta_{m}(1/2+it)$ and the corresponding Dirichlet polynomial.
We are interested in the exact value distribution of $\eta_{m}(1/2 + it)$ and $S_{m}(t)$. 
Here our aim is to establish a theorem for $\eta_{m}(1/2+it)$ and $S_{m}(t)$ similar to the results of Jutila \cite{Ju1983}, 
Radziwi\sl\sl\mbox{} \cite{Ra2011}, and Soundararajan \cite{SM2009} on the large deviation of the Riemann zeta-function. 
The motivation of this study in the present paper is to search for the exact bound of $\eta_{m}(1/2+it)$.

We define the set $\T_{m}(T, X, V)$ by
\begin{align*}
\l\{ t \in [T, 2T] \; \middle| \; 
\bigg|\eta_{m}(1/2 + it) 
- i^{m}\sum_{2 \leq n \leq X}\frac{\Lam(n)}{n^{\frac{1}{2} + it}(\log{n})^{m+1}} 
- Y_{m}(1/2+it) \bigg|
>  V \r\}.
\end{align*}
We obtain the following result which evaluates the difference between $\eta_{m}(1/2+it)$ and the corresponding Dirichlet polynomial.

\begin{theorem}	\label{MLFmGM}
Let $m$ be a positive integer, and let $T$, $X$ be large with $X^{135} \leq T$. 
If $V$ satisfies the inequality $2(\log{X})^{-m} \leq V 
\leq c_{0}(\log{T})^{\frac{m}{2m+1}} (\log{X})^{-\frac{2m^{2} + 2m}{2m+1}}$,  
then we have
\begin{align}	\label{Rmk1}
\frac{1}{T}\meas(\T_{m}(T, X, V))
\ll \exp\l( -  \frac{m}{4(m+1)}V^2 (\log{X})^{2m} \l(1 - \frac{C}{\log{X}}\r) \r).
\end{align}
If $V$ satisfies $c_{0}(\log{T})^{\frac{m}{2m+1}} (\log{X})^{-\frac{2m^{2} + 2m}{2m+1}} \leq V \leq \log{T} / (\log{X})^{m+1}$, 
then we have
\begin{align}	\label{Rmk3}
\frac{1}{T}\meas(\T_{m}(T, X, V))
\ll \exp\l(-c_1 V^{\frac{1}{m+1}} (\log{T})^{\frac{m}{m+1}} \r).
\end{align}
Moreover, if the Riemann Hypothesis is true, then we have
\begin{multline}	\label{Rmk2}
\frac{1}{T}\meas(\T_{m}(T, X, V))\\
\ll \exp\l(-c_2 V^{\frac{1}{m+1}} (\log{T})^{\frac{m}{m+1}} 
\log\l( e\frac{V^{\frac{2m+1}{2m+2}}(\log{X})^{m}}{(\log{T})^{\frac{m}{2m+2}}} \r)\r)
\end{multline}
for $(\log{T})^{\frac{m}{2m+1}} (\log{X})^{-\frac{2m^{2} + 2m}{2m+1}} \leq V \leq \log{T} / (\log{X})^{m+1}$.
Here the numbers $c_0$, $c_1$, $c_2$, $C$ are some absolute positive constants.
\end{theorem}

This theorem can be applied to the value distribution of $\eta_{m}(s)$ on the critical line.
For example, we can obtain the following results from this theorem.

\begin{corollary}	\label{EVS1}
Let $T$, $V$ be large numbers.
If $V \leq (\log{T})^{1/3}(\log\log{{T}})^{-4/3}$, then we have
\begin{align}	\label{VDS11}
\frac{1}{T}\meas\l\{ t \in [T, 2T] \mid |S_{1}(t)| > V \r\}
\ll \exp\l(-c_{5} V^{2} (\log{V})^2\r).
\end{align}
If $V \geq (\log{T})^{1/3}(\log\log{{T}})^{-4/3}$, then we have
\begin{align}	\label{VDS12}
\frac{1}{T}\meas\l\{ t \in [T, 2T] \mid |S_{1}(t)| > V \r\}
\ll \exp\l(-c_{6} \sqrt{V \log{T}}\r).
\end{align}
Here $c_5$, $c_6$ are some absolute positive constants. 
\end{corollary}

\begin{corollary}	\label{EVeta}
Assume the Riemann Hypothesis.
Let $m$ be a positive integer, and let $T$, $V$ be numbers with $T, V \geq T_{0}(m)$, 
where $T_{0}(m)$ is a sufficiently large number depending only on $m$.
Then, if $V \leq (\log{T})^{\frac{m}{2m+1}}(\log{\log{T}})^{-\frac{2m^2+2m}{2m+1}}$, 
we have
\begin{align}	\label{CVDeta1}
\frac{1}{T}\meas\l\{ t \in [T, 2T] \mid |\eta_{m}(1/2+it)| > V \r\}
\ll \exp(-c_{7}V^2 (\log{V})^{2m}).
\end{align}
Moreover, if $V \geq (\log{T})^{\frac{m}{2m+1}}(\log{\log{T}})^{-\frac{2m^2+2m}{2m+1}}$, then we have
\begin{multline}	\label{CVDeta2}
\frac{1}{T}\meas\l\{ t \in [T, 2T] \mid |\eta_{m}(1/2+it)| > V \r\}\\
\ll \exp\l( -c_{8} V^{\frac{1}{m+1}}(\log{T})^{\frac{m}{m+1}} 
\log\l( e\frac{V^{\frac{2m+1}{2m+2}} (\log{V})^{m}}{(\log{T})^{\frac{m}{2m+2}}} \r) \r).
\end{multline}
Here $c_7$, $c_8$ are some absolute positive constants.
\end{corollary}

These assertions can be obtained by the following argument.
Now, we see that $\sum_{2\leq n \leq V} \frac{\Lam(n)}{n^{1/2+it}(\log{n})^{m+1}} \ll_{m} \frac{V^{1/2}}{(\log{V})^{m+1}}$. 
Hence, for sufficiently large $V$, we find that
\begin{align*}
\meas\l\{ t \in [T, 2T] \mid |S_{1}(t)| > V \r\}
\leq \meas(\T_{1}(T, V, V/2))
\end{align*}
unconditionally, and that
\begin{align*}
\meas\l\{ t \in [T, 2T] \mid |\eta_{m}(1/2+it)| > V \r\}
\leq \meas(\T_{m}(T, V, V/2))
\end{align*}
under the Riemann Hypothesis.
Further, the estimate $S_{1}(t) \ll \log{t}$ holds unconditionally, 
and the estimate $\eta_{m}(1/2+it)\ll_{m} \log{t}/(\log{\log{t}})^{m+1}$ holds under the Riemann Hypothesis.
By these inequalities and Theorem \ref{MLFmGM}, we can obtain Corollary \ref{EVS1} and Corollary \ref{EVeta}.

It could be expected that the function $\sqrt{V \log{T}}$ in the exponential on the right hand side of \eqref{VDS12} 
is sharp as an unconditional result by the following discussion.
Actually, if there is a function $\omega(T, V)$ with $\lim_{T \rightarrow +\infty}\omega(T, V) = +\infty$ 
or $\lim_{V \rightarrow +\infty}\omega(T, V) = +\infty$
such that the left hand side of \eqref{VDS12} is $\ll \exp(-\omega(T, V) \sqrt{V \log{T}})$, 
then the Lindel\"of Hypothesis holds.
Moreover, estimate \eqref{VDS12} matches the well known inequality $S_{1}(t) \ll \log{t}$.

We are also interested in that estimates \eqref{VDS11}, \eqref{CVDeta1} hold in how large range of $V$. 
If the estimates hold for any large $V$, then we have $\eta_{m}(1/2+it) \ll_{m} \sqrt{\log{t}} / (\log{\log{t}})^{m}$.
Although the necessary condition of this implication is rather strong, the author guesses that it could be true.
Hence the author expects the inequality for $\eta_{m}(1/2+it)$ could be also true.

\section{\textbf{Proofs of Theorem \ref{Main_Prop} and Theorem \ref{RELZZ}}}

In this section, we prove Theorem \ref{Main_Prop} and Theorem \ref{RELZZ}.
First, we prepare some auxiliary formulas.

\begin{lemma}	\label{BFQm}
Let $m$ be a positive integer, and let $t > 0$.
Then, for any $\s \geq 1/2$, we have
\begin{multline*}
\eta_{m}(\s + it)
= \frac{i^{m}}{(m-1)!}\int_{\s}^{\infty}(\a - \s)^{m-1}\log{\zeta(\a + it)}d\a\\
+ 2\pi \sum_{k = 0}^{m-1}\frac{i^{m-1-k}}{(m-k)! k!}\us{\b > \s}{\sum_{0 < \gamma < t}}(\b - \s)^{m-k}(t - \gamma)^{k}.
\end{multline*}
\end{lemma}

\begin{proof}
In view of our choice of the branch of $\log{\zeta(s)}$, 
it suffices to show this lemma in the case $t$ is not the ordinate of zeros of $\zeta(s)$.
We show this lemma by induction on $m$.
When $m = 0$, by using Littlewood's lemma (cf. (9.9.1) in \cite{T}), it holds that
\begin{multline}	\label{LLZ}
i\int_{0}^{t}\log{\zeta(\s + it')}d t' - \int_{\s}^{\infty}\log{\zeta(\a)}d\a\\
= -\int_{\s}^{\infty}\log{\zeta(\a + it)}d\a + 2\pi i\int_{\s}^{\infty}N(\a, t)d\a.
\end{multline}
Here $N(\s, t)$ indicates the number of zeros $\rho = \b + i\gamma$ of the Riemann zeta-function 
with $\b \geq \s$, $0 < \gamma < t$
counted with multiplicity.
We see that
\begin{align*}
\int_{\s}^{\infty}N(\a, t)d\a
= \int_{\s}^{\infty}\us{\b > \a}{\sum_{0 < \gamma < t}}1d\a
= \us{\b > \s}{\sum_{0 < \gamma < t}} \int_{\s}^{\b}d\a
= \us{\b > \s}{\sum_{0 < \gamma < t}}(\b - \s).
\end{align*}
Therefore, by this formula and the definition of $\eta_{m}(s)$, we have
\begin{align*}
\eta_{1}(\s + it)
= i\int_{\s}^{\infty}\log{\zeta(\a+it)}d\a + 2\pi \us{\b > \s}{\sum_{0 < \gamma < t}}(\b - \s),
\end{align*}
which is the assertion of this lemma in the case $m=1$.

Next we show this lemma in the case $m \geq 2$.
Assume that the assertion of this lemma is true at $m-1$.
Then, we find that
\begin{align} \nonumber
&\int_{0}^{t}\eta_{m-1}(\s + it')dt'\\ \nonumber
&= \int_{0}^{t}\frac{i^{m-1}}{(m-2)!}\int_{\s}^{\infty}(\a - \s)^{m-2}\log{\zeta(\a+it')}d\a dt' \\ \nonumber
&\qquad \qquad \quad + 2\pi \sum_{k = 0}^{m-2}\frac{i^{m-2-k}}{(m-1-k)! k!}
\int_{0}^{t}\us{\b > \s}{\sum_{0 < \gamma < t'}}(\b - \s)^{m-1-k}(t' - \gamma)^{k}dt' \\ \nonumber
&= \frac{i^{m-1}}{(m-2)!}\int_{\s}^{\infty}(\a - \s)^{m-2}\int_{0}^{t}\log{\zeta(\a + it')}d t' d\a\\ \label{BFQmp1}
&\qqqquad + 2\pi \sum_{k = 1}^{m-1}\frac{i^{m-1-k}}{(m-k)!k!}
\us{\b > \s}{\sum_{0 < \gamma < t}}(\b - \s)^{m-k}(t - \gamma)^{k}.
\end{align}
Note that the exchange of integration of the first term in the second equation 
is guaranteed by the absolute convergence of the integral.
Applying formula \eqref{LLZ}, we find that
\begin{align*}
&\frac{i^{m-1}}{(m-2)!}\int_{\s}^{\infty}(\a - \s)^{m-2}\int_{0}^{t}\log{\zeta(\a + it')}d t' d\a\\
&= \frac{i^{m}}{(m - 1)!}\int_{\s}^{\infty}(\a - \s)^{m-1}\log{\zeta(\a + it)}d\a - c_{m}(\s)\\
&\qqqquad \qqquad + 2\pi \frac{i^{m-1}}{(m-1)!}\int_{\s}^{\infty}(\a - \s)^{m-1} N(\a, t)d\a,
\end{align*}
and that
\begin{align*}
\int_{\s}^{\infty}(\a - \s)^{m-1} N(\a, t)d\a
= \us{\b > \s}{\sum_{0 < \gamma < t}}\int_{\s}^{\b}(\a - \s)^{m-1}d\a
= \frac{1}{m}\us{\b > \s}{\sum_{0 < \gamma < t}}(\b - \s)^{m}.
\end{align*}
Hence, by these formulas, \eqref{BFQmp1}, and the definition of $\eta_{m}(s)$, we obtain
\begin{multline*}
\eta_{m}(\s + it)
= \frac{i^{m}}{(m-1)!}\int_{\s}^{\infty}(\a - \s)^{m-1}\log{\zeta(\a + it)}d\a\\
+ 2\pi \sum_{k = 0}^{m-1}\frac{i^{m-1-k}}{(m-k)! k!}\us{\b > \s}{\sum_{0 < \gamma < t}}(\b - \s)^{m-k}(t - \gamma)^{k},
\end{multline*}
which completes the proof of this lemma.
\end{proof}

\begin{lemma}	\label{GUE}
Let $m$, $d$ be a nonnegative integers with $d \leq D = D(f)$.
Let $\s \geq 1/2$, $t \geq 14$ be not the ordinate of zeros of $\zeta(s)$. 
Set $X \geq 3$ be a real parameter.
Then for any zero $\rho = \b + i\gamma$, we have
\begin{align*}
\frac{U_{m}((s - \rho)\log{X})}{(\log{X})^{m}}
\ll_{f, d} \frac{X^{(1+1/H)(\b - \s)} + X^{\b - \s}}{|t - \gamma|(\log{X})^{m+1}}
\min_{0 \leq l \leq d}\l\{\l(\frac{H}{|t - \gamma|\log{X}} \r)^{l}\r\}.
\end{align*}
\end{lemma}

\begin{proof}
By the definition of $U_{m}(z)$, we have
\begin{multline}	\label{FUm}
\frac{U_{m}((s - \rho)\log{X})}{(\log{X})^{m}}\\
= \frac{1}{m!}\int_{\s - \b}^{\infty}\frac{(\a - \s + \b)^{m}}{\a + i(t - \gamma)}\l(\int_{0}^{\infty}
u_{f, H}(x)e^{-(\a + i(t - \gamma))\log{X}\log{x}}dx \r)d\a.
\end{multline}
Since $u_{f, H}$ belongs to $C^{D-2}([0, \infty))$ 
and is a $C^{D}([e, e^{1+1/H}])$-function and supported on $[e, e^{1+1/H}]$, 
for $0 \leq d \leq D-1$, we see that
\begin{multline}	\label{IIPu}
\int_{0}^{\infty}u_{f, H}(x)e^{-(\a + i(t - \gamma))\log{X}\log{x}}dx\\
= \int_{e}^{e^{1+1/H}}\frac{u_{f, H}^{(d)}(x) x^{d - (\a + i(t - \gamma))\log{X}}}
{\prod_{l=1}^{d}\{(\a + i(t  - \gamma))\log{X} - l\}}dx.
\end{multline}
Here the estimate $u_{f, H}^{(d)}(x) \ll_{f, d} H^{d+1}$ holds 
on $x \in [e, e^{1+1/H}]$ for $0 \leq d \leq D$.
By this estimate and \eqref{IIPu}, we have
\begin{multline*}
\int_{0}^{\infty}u_{f, H}(x)e^{-(\a + i(t - \gamma))\log{X}\log{x}}dx\\
\ll_{f, d} (X^{-(1+\frac{1}{H})\a} + X^{-\a})\min_{0 \leq l \leq d}\l\{\l(\frac{H}{|t - \gamma|\log{X}}\r)^{l}\r\}
\end{multline*}
for $0 \leq d \leq D - 1$.
Moreover, by \eqref{IIPu}, we find that
\begin{align*}
&\int_{0}^{\infty}u_{f, H}(x)e^{-(\a + i(t - \gamma))\log{X}\log{x}}dx\\
&= \l[ \frac{u_{f, H}^{(D-1)}(x) x^{D - (\a + i(t - \gamma))\log{X}}}
{\prod_{l=1}^{D}\{(\a + i(t  - \gamma))\log{X} - l\}} \r]_{x=e}^{x=e^{1+1/H}}\\
&\qqqquad \qqquad + \int_{e}^{e^{1+1/H}}\frac{u_{f, H}^{(D)}(x) x^{D - (\a + i(t - \gamma))\log{X}}}
{\prod_{l=1}^{D}\{(\a + i(t  - \gamma))\log{X} - l\}}dx\\
&\ll_{f, D} (X^{-(1+\frac{1}{H})\a} + X^{-\a})\l(\frac{H}{|t - \gamma|\log{X}}\r)^{D}.
\end{align*}
By these estimates and \eqref{FUm}, for $0 \leq d \leq D$, we have
\begin{align*}
&\frac{U_{m}((s - \rho)\log{X})}{(\log{X})^{m}}\\
&\ll_{f, d} \frac{1}{|t - \gamma|}\min_{0 \leq l \leq d}\l\{\l(\frac{H}{|t - \gamma|\log{X}}\r)^{l}\r\}
\int_{\s - \b}^{\infty}(\a - \s + \b)^{m}(X^{-\a(1+1/H)} + X^{-\a})d\a\\
&\ll\frac{X^{(1+1/H)(\b - \s)} + X^{\b - \s}}{|t - \gamma|(\log{X})^{m+1}}
\min_{0 \leq l \leq d}\l\{\l(\frac{H}{|t - \gamma|\log{X}} \r)^{l}\r\},
\end{align*}
which completes the proof of this lemma.
\end{proof}

\begin{lemma}	\label{EUE}
Let $m$ be a nonnegative integer, 
and let $\s \geq 1/2$, $t \geq 14$.
Set $X \geq 3$ be a real parameter.
Then, for a zero $\rho = \b + i\gamma$ with $|t - \gamma| \leq 1 / \log{X}$, we have
\begin{multline}	\label{NZUm}
\frac{U_{m}((s - \rho)\log{X})}{(\log{X})^{m}}\\
= \l\{
\begin{array}{cl}
-(\rho - s)^{m}\log((s - \rho)\log{X}) + \d{O\l( \frac{1}{(\log{X})^{m}} \r)}	& \text{if \; $|s - \rho| \leq 1/\log{X}$, }\\[5mm]
\d{O\l( \frac{X^{(1+1/H)(\b - \s)} + X^{\b - \s}}{(\log{X})^{m}} \r)}				& \text{if \; $|s - \rho| > 1 / \log{X}$.}
\end{array}
\r.
\end{multline}
In particular, for any positive integer $m$, we have
\begin{align}	\label{SCUm}
\frac{U_{m}((s - \rho)\log{X})}{(\log{X})^{m}}
\ll \frac{X^{(1+1/H)(\b - \s)} + X^{\b - \s}}{(\log{X})^{m}}.
\end{align}
Here, the above implicit constants are absolute.
\end{lemma}

\begin{proof}
In view of our definition of $U_{m}(z)$ and $\log{z}$, it suffices to show this lemma in the case that 
$t$ is not equal to the ordinate of zeros of $\zeta(s)$.
First, we consider the case $\s \geq \b + 1/\log{X}$.
Then we see that
\begin{align*}
&\frac{U_{m}((s - \rho)\log{X})}{(\log{X})^{m}}\\
&= \frac{1}{m!}\int_{0}^{\infty}u_{f, H}(x)\int_{\s - \b}^{\infty}
(\a - \s + \b)^{m}\frac{e^{-(\a + i(t - \gamma))\log{X}\log{x}}}{\a + i(t - \gamma)}d\a dx\\
&\ll \frac{1}{m!}\int_{e}^{e^{1+1/H}}u_{f, H}(x)\int_{\s - \b}^{\infty}(\a - \s + \b)^{m-1}e^{-\a\log{X}\log{x}}d\a dx
\ll \frac{X^{\b-\s}}{(\log{X})^{m}}.
\end{align*}
Next, we consider the case $|\s - \b| \leq 1/\log{X}$.
Put $s_1 = \b + 1/\log{X} + it$.
Then we can write
\begin{multline*}
\frac{U_{m}((s - \rho)\log{X})}{(\log{X})^{m}} =\\
\frac{1}{m!}\int_{e}^{e^{1+1/H}}\frac{u_{f, H}(x)}{(\log{X}\log{x})^{m}}\int_{(s - \rho)\log{X}\log{x}}^{(s_1 - \rho)\log{X}\log{x}}
(w - (s - \rho)\log{X}\log{x})^{m}\frac{e^{-w}}{w}dwdx \\
+ \frac{U_{m}((s_1 - \rho)\log{X})}{(\log{X})^{m}}.
\end{multline*}
By the estimate in the previous case of $\s \geq \b + 1/\log{X}$, the second term on the right hand side is
$\ll (\log{X})^{-m}$.

We consider the first term on the right hand side.
By the Taylor expansion, it holds that
\begin{align*}
&\int_{(s - \rho)\log{X}\log{x}}^{(s_1 - \rho)\log{X}\log{x}}(w - (s - \rho)\log{X}\log{x})^{m}\frac{e^{-w}}{w}dw\\
&= \int_{(s - \rho)\log{X}\log{x}}^{(s_1 - \rho)\log{X}\log{x}}\frac{(w - (s - \rho)\log{X}\log{x})^{m}}{w}dw+\\
&\qqquad+\sum_{n = 1}^{\infty}\frac{(-1)^n}{n!}\int_{(s - \rho)\log{X}\log{x}}^{(s_1 - \rho)\log{X}\log{x}}
(w - (s - \rho)\log{X}\log{x})^{m}w^{n-1}dw.
\end{align*}
For the first integral on the above, by the binomial expansion, we find that 
\begin{align*}
&\frac{1}{(\log{X}\log{x})^{m}}\int_{(s - \rho)\log{X}\log{x}}^{(s_1 - \rho)\log{X}\log{x}}\frac{(w - (s - \rho)\log{X}\log{x})^{m}}{w}dw\\
&= \frac{1}{(\log{X}\log{x})^m}\sum_{k = 0}^{m}
\begin{pmatrix}
m\\k
\end{pmatrix}
((\rho - s)\log{X}\log{x})^{m - k}\int_{(s - \rho)\log{X}\log{x}}^{(s_1 - \rho)\log{X}\log{x}}w^{k-1}dw\\
&= (\rho - s)^{m}\log{\l( \frac{s_1 - \rho}{s - \rho} \r)}
+\sum_{k = 1}^{m}\frac{1}{k}
\begin{pmatrix}
m\\k
\end{pmatrix}
(\rho - s)^{m - k}\l\{(s_1 - \rho)^{k} - (s - \rho)^{k} \r\}\\
&= -(\rho - s)^{m}\log{((s - \rho)\log{X})} + O\l(\frac{4^{m}}{(\log{X})^{m}}\r).
\end{align*}
On the terms for $n \geq 1$, we see that
\begin{align*}
&\frac{1}{(\log{X}\log{x})^{m}}\int_{(s - \rho)\log{X}\log{x}}^{(s_1 - \rho)\log{X}\log{x}}
(w - (s - \rho)\log{X}\log{x})^{m}w^{n-1}dw
\ll \frac{2^{m}6^{n-1}}{(\log{X})^{m}}.
\end{align*}
Therefore, by the above calculations, when $|\s - \b| \leq 1/\log{X}$, we obtain
\begin{align*}
\frac{U_{m}((s - \rho)\log{X})}{(\log{X})^{m}}
= -\frac{1}{m!}(\rho - s)^{m}\log{((s - \rho)\log{X})} + O\l(\frac{1}{(\log{X})^{m}}\r).
\end{align*}

Finally, we consider the case $\s \leq \b - 1/\log{X}$.
Put $s_2 = \b - 1/\log{X} + it$. Then we can write
\begin{multline*}
\frac{U_{m}((s - \rho)\log{X})}{(\log{X})^{m}} =\\
\frac{1}{m!}\int_{e}^{e^{1+1/H}}\frac{u_{f, H}(x)}{(\log{X}\log{x})^{m}}\int_{(s - \rho)\log{X}\log{x}}^{(s_2 - \rho)\log{X}\log{x}}
(w - (s - \rho)\log{X}\log{x})^{m}\frac{e^{-w}}{w}dwdx \\
+ \frac{U_{m}((s_2 - \rho)\log{X})}{(\log{X})^{m}}.
\end{multline*}
Now by using the result of the previous case, we have
\begin{align*}
\frac{U_{m}((s_2 - \rho)\log{X})}{(\log{X})^{m}}
&= -\frac{1}{m!}(\rho - s_2)^{m}\log{((s_2 - \rho)\log{X})} + O\l(\frac{1}{(\log{X})^{m}}\r)\\
&\ll \frac{1}{(\log{X})^{m}}.
\end{align*}
On the other hand, from the definition of $U_{m}(z)$, we see that the first term on the right hand side is
\begin{align*}
&=\frac{1}{m!}\int_{e}^{e^{1+1/H}}u_{f, H}(x)\int_{\s - \b}^{\frac{-1}{\log{X}}}
(\a - \s + \b)^{m}\frac{e^{-(\a + i(t - \gamma))\log{X}\log{x}}}{\a + i(t - \gamma)}d\a dx\\
&\ll \frac{\log{X}}{m!}\int_{e}^{e^{1+1/H}}u_{f, H}(x)\int_{\s - \b}^{\infty}(\a - \s + \b)^{m}e^{-\a\log{X}\log{x}}d\a dx
\ll \frac{X^{(1+1/H)(\b - \s)}}{(\log{X})^{m}}.
\end{align*}
From the above calculations, we obtain
\begin{multline*}
\frac{U_{m}((s - \rho)\log{X})}{(\log{X})^{m}}\\
= \l\{
\begin{array}{cl}
-(\rho - s)^{m}\log((s - \rho)\log{X}) + O\l( \frac{1}{(\log{X})^{m}} \r)	& \text{if \; $|\s - \b| \leq 1/\log{X}$, }\\[5mm]
O\l( \frac{X^{(1+1/H)(\b - \s)} + X^{\b - \s}}{(\log{X})^{m}} \r)			& \text{if \; $|\s - \b| > 1 / \log{X}$.}
\end{array}
\r.
\end{multline*}
Now, from the condition $|t - \gamma| \leq 1/\log{X}$, the formula where $|\s - \b|$ is replaced by $|s- \rho|$ also holds.
Hence, we complete the proof of the estimate \eqref{NZUm}.

Moreover, we can obtain the estimate \eqref{SCUm} from \eqref{NZUm} 
since, for $m \in \ZZ_{\geq 1}$, the inequality 
$\frac{1}{m!}(s - \rho)^{m}\log{((s - \rho)\log{X})} \ll (\log{X})^{-m}$ holds for $|s - \rho| \leq 1 / \log{X}$.
Thus, we obtain this lemma.
\end{proof}

\begin{proposition}	\label{BPI}
Let $m$ be a nonnegative integer.
Then, for $\s \geq 1/2$, $t \geq 14$ we have
\begin{multline*}
\eta_{m}(s)
= i^{m}\sum_{2 \leq n \leq X^{1 + 1/H}}\frac{\Lam(n)v_{f, H}(e^{\log{n} / \log{X}})}{n^{s}(\log{n})^{m + 1}}
-\frac{i^{m}}{(\log{X})^{m}} \sum_{\rho}U_{m}((s - \rho)\log{X})\\
+ 2\pi \sum_{k = 0}^{m-1}\frac{i^{m-1-k}}{(m-k)! k!}\us{\b > \s}{\sum_{0 < \gamma < t}}(\b - \s)^{m-k}(t - \gamma)^{k}
+O\l( \frac{X^{2(1 - \s)} + X^{1-\s}}{t(\log{X})^{m + 1}} \r).
\end{multline*}
Here if $m = 0$, then we regard that the third term on the right hand side is zero.
\end{proposition}

\begin{proof}
In view of our definition of $U_{m}(z)$ and $\log{\zeta(s)}$, it suffices to show this lemma in the case that 
$t$ is not equal to the ordinate of zeros of $\zeta(s)$.
First, we prove this proposition in the case $m = 0$.
The proof is the almost same as the proof of Theorem 1 in \cite{GHK2007} 
(see also the proof of Lemma 1 in \cite{BH1995}, if necessary). 
Hence, we only write the rough proof in this case.
Let $\tilde{u}(s)$ be the Mellin transform of $u_{f, H}$, that is, $\tilde{u}(s) := \int_{0}^{\infty}u_{f, H}(x)x^{s - 1}dx$.
Since the functions $v_{f, H}(x)$ and $\tilde{u}(s + 1)/s$ are Mellin transforms, we find that, 
for any complex number $z$ with $\Re(z) \geq 1/2$, 
\begin{align*}
\sum_{n = 1}^{\infty}\frac{\Lam(n)}{n^{z}}v_{f, H}\l( e^{\log{n}/\log{X}} \r)
&=\frac{1}{2\pi i}\sum_{n = 1}^{\infty}\frac{\Lam(n)}{n^{z}}
\int_{c-i\infty}^{c+i\infty}\frac{\tilde{u}(w+1)}{w}n^{-w/\log{X}}dw\\
&= -\frac{1}{2\pi i}\int_{\log{X} - i\infty}^{\log{X} + i\infty}
\frac{\zeta'}{\zeta}\l(z + \frac{w}{\log{X}}\r)\frac{\tilde{u}(w+1)}{w}dw.
\end{align*}
By this formula, for $\Re(z) \geq 1/2$, $\Im(z) \geq 14$, we have
\begin{multline*}
\sum_{n \leq X^{1+1/H}}\frac{\Lam(n)}{n^{z}}v_{f, H}\l( e^{\log{n}/\log{X}} \r)\\
= -\frac{\zeta'}{\zeta}(z)  - \sum_{\rho}\frac{1}{\rho - z}\tilde{u}(1 + (\rho - z)\log{X})
+O\l( \frac{X^{2(1 - \Re(z))} + X^{1 - \Re(z)}}{\Im(z)} \r).
\end{multline*}
Integrating both sides with respect to $z$ from $\infty + it$ to $\s + it \ (= s)$ , we obtain
\begin{multline}	\label{BPIm0}
\log{\zeta(s)} 
= \sum_{2 \leq n \leq X^{1+1/H}}\frac{\Lam(n)}{n^{s}\log{n}}v_{f, H}\l( e^{\log{n}/\log{X}} \r)\\
- \sum_{\rho}U_{0}((s - \rho)\log{X}) 
+ O\l( \frac{X^{2(1 - \s)} + X^{1 - \s}}{t \log{X}} \r).
\end{multline}
Therefore, this theorem holds in the case $m = 0$.

Next we show this proposition for $m \geq 1$. By Lemma \ref{BFQm}, it suffices to show that
\begin{multline}	\label{BPIm1}
\frac{i^{m}}{(m - 1)!}\int_{\s}^{\infty}(\a - \s)^{m-1}\log{\zeta(\a + it)}d\a\\
= i^{m}\sum_{2 \leq n \leq X^{1 + 1/H}}\frac{\Lam(n)v_{f, H}(e^{\log{n} / \log{X}})}{n^{s}(\log{n})^{m + 1}}
-\frac{i^{m}}{(\log{X})^{m}} \sum_{\rho}U_{m}((s - \rho)\log{X})\\
+O\l( \frac{X^{2(1 - \s)} + X^{1-\s}}{t(\log{X})^{m + 1}} \r).
\end{multline}
Here, by using formula \eqref{BPIm0}, the left hand side on the above is
\begin{multline}	\label{BPI1}
= i^{m}\sum_{2 \leq n \leq X^{1 + 1/H}}\frac{\Lam(n)v_{f, H}(e^{\log{n} / \log{X}})}{n^{s}(\log{n})^{m + 1}} -\\
-\frac{i^{m}}{(m - 1)!} \int_{\s}^{\infty}\sum_{\rho}(\a - \s)^{m-1}U_{0}((\a + it - \rho)\log{X})d\a
+O\l( \frac{X^{2(1 - \s)} + X^{1-\s}}{t(\log{X})^{m + 1}} \r).
\end{multline}
In the following, we will change the above sum and integral, and it is guaranteed by
\begin{align*}
\sum_{\rho}\int_{\s}^{\infty}|(\a - \s)^{m-1}U_{0}((\a + it - \rho)\log{X})|d\a
< +\infty.
\end{align*}
This convergence can be obtained by Lemma \ref{GUE}.
Further, a little calculus shows that
\begin{multline}	\label{BPI2}
\frac{i^{m}}{(m-1)!}\int_{\s}^{\infty}(\a - \s)^{m-1}U_{0}((\a + it - \rho)\log{X})d\a\\
= \frac{i^{m}}{(\log{X})^m}U_{m}((s - \rho)\log{X}).
\end{multline}
Hence, by \eqref{BPI1}, \eqref{BPI2}, we obtain formula \eqref{BPIm1}, and this completes the proof of this proposition.
\end{proof}

\begin{proof}[Proof of Theorem \ref{Main_Prop}]
We can immediately obtain estimate \eqref{ESRm} by Proposition \ref{BPI}, 
Lemma \ref{GUE}, and Lemma \ref{EUE}. Now we prove estimate \eqref{ESRm2} under the Riemann Hypothesis. 
It suffices to show
\begin{align}	\label{MPE1}
\sum_{\frac{1}{\log{X}} < |t - \gamma| \leq \frac{H}{\log{X}}}\frac{1}{|t - \gamma|}
&\ll \log{t}\l(\frac{\log{X}}{\log{\log{t}}} + \log{H} \r),
\end{align}
and 
\begin{align}	\label{MPE2}
\sum_{|t - \gamma| > \frac{H}{\log{X}}}\frac{H}{(t - \gamma)^2\log{X}}
\ll \log{t} \times \l( \frac{\log{X}}{H\log{\log{t}}} + 1 \r)
\end{align}
under the Riemann Hypothesis.
Assuming the Riemann Hypothesis, the following estimate (cf. Lemma 13.19 in \cite{MV})
\begin{align}	\label{BINESZ}
\tilde{N}\l( t, \frac{1}{\log{\log{t}}} \r)
\ll \frac{\log{t}}{\log{\log{t}}}
\end{align}
holds for $t \geq 5$.
By this estimate, for any $1 \leq H \leq \frac{t}{2}$, we find that
\begin{align*}
&\sum_{\frac{1}{\log{X}} < |t - \gamma| \leq \frac{H}{\log{X}}}\frac{1}{|t - \gamma|}
\leq \sum_{k = 0}^{[(H-1)\frac{\log{\log{t}}}{\log{X}}]}
\sum_{\frac{1}{\log{X}} + \frac{k}{\log{\log{t}}} < |t - \gamma| \leq \frac{1}{\log{X}} + \frac{k+1}{\log{\log{t}}}}
\frac{1}{|t - \gamma|}\\
&\ll \log{t}\sum_{k = 0}^{[(H - 1)\frac{\log{\log{t}}}{\log{X}}]}\frac{1}{\frac{\log{\log{t}}}{\log{X}} + k}
\leq \log{t}\l( \frac{\log{X}}{\log{\log{t}}} 
+ \int_{0}^{(H - 1)\frac{\log{\log{t}}}{\log{X}}}\frac{du}{\frac{\log{\log{t}}}{\log{X}} + u} \r)\\
&= \log{t}\l( \frac{\log{X}}{\log{\log{t}}} + \log{H} \r),
\end{align*}
and that
\begin{align*}
&\sum_{|t - \gamma| > \frac{H}{\log{X}}}\frac{H}{(t - \gamma)^2 \log{X}}
= \sum_{\frac{H}{\log{X}} < |t - \gamma| \leq \frac{t}{2}}\frac{H}{(t - \gamma)^2 \log{X}} + O\l( \frac{H}{t\log{X}} \r)\\
&\leq \sum_{k = 0}^{[\frac{t\log{\log{t}}}{2}]}
\sum_{\frac{H}{\log{X}} + \frac{k}{\log{\log{t}}} < |t - \gamma| \leq \frac{H}{\log{X}} + \frac{k+1}{\log{\log{t}}}}
\frac{H}{(t - \gamma)^2 \log{X}} + O\l( \frac{H}{t\log{X}} \r)\\
&\ll H\log{\log{t}}\frac{\log{t}}{\log{X}}\sum_{k = 0}^{[\frac{t\log{\log{t}}}{2}]}
\frac{1}{\l( k + \frac{H\log{\log{t}}}{\log{X}} \r)^2} + \frac{H}{t\log{X}}\\
&\leq H\log{\log{t}}\frac{\log{t}}{\log{X}}\l( \l(\frac{\log{X}}{H\log{\log{t}}}\r)^2 
+ \int_{0}^{\infty}\frac{du}{\l( u + \frac{H\log{\log{t}}}{\log{X}} \r)^2}\r) + \frac{H}{t\log{X}}\\
&\ll \log{t}\l( \frac{\log{X}}{H\log{\log{t}}} + 1 \r).
\end{align*}
Hence, we obtain estimates \eqref{MPE1}, \eqref{MPE2}.
\end{proof}

Next we prove Theorem \ref{RELZZ}. Here, we prepare a standard conditional formula.

\begin{lemma}	\label{BLFRUR}
Assume the Riemann Hypothesis.
Then, for $t \geq 14$, $\frac{1}{2} \leq \s \leq \frac{1}{2} + \frac{1}{\log{\log{t}}}$,
\begin{align}	\label{BLF}
\frac{\zeta'}{\zeta}(s)
= \sum_{|t - \gamma| \leq 1 / \log{\log{t}}}\frac{1}{s - \rho}
+ O\l(\log{t}\r).
\end{align}
\end{lemma}

\begin{proof}
This lemma is Lemma 13.20 in \cite{MV}.
\end{proof}

\begin{proof}[Proof of Theorem \ref{RELZZ}]
Let $t \geq 14$ and $X$ be a real parameter with $\log{t} \leq X \leq t$.
By using Theorem \ref{Main_Prop}, we have
\begin{align*}
P_{f}(1/2+it, X)
= \log{\zeta(1/2 + it )} 
- \hspace{-0.7mm} \sum_{|t - \gamma| \leq \frac{1}{\log{X}}}\log{\l( |t - \gamma|\log{X} \r)}
+O_{f}\l(\frac{\log{t}}{\log{\log{t}}} \r).
\end{align*}
By integrating the both sides of \eqref{BLF}, we obtain
\begin{multline*}
\log{\zeta\l( \frac{1}{2} + it \r)} - \log{\zeta\l( \frac{1}{2} + \frac{1}{\log{\log{t}}} + it \r)}\\
= \sum_{|t - \gamma| \leq \frac{1}{\log{\log{t}}}}\log{\l(|t - \gamma|\log{\log{t}}\r)}
+O\l( \frac{\log{t}}{\log{\log{t}}} \r),
\end{multline*}
and by using estimate (13.44) in \cite{MV}, we obtain
\begin{align*}
\log{ \zeta\l( \frac{1}{2} + \frac{1}{\log{\log{t}}} + it \r)}
\ll \frac{\log{t}}{\log{\log{t}}}.
\end{align*}
Hence, we obtain
\begin{align*}
&P_{f}(1/2+it, X)=\\
& \sum_{|t - \gamma| \leq \frac{1}{\log{\log{t}}}}\log{\l(|t - \gamma|\log{\log{t}}\r)} 
- \sum_{|t - \gamma| \leq \frac{1}{\log{X}}}\log{\l( |t - \gamma|\log{X} \r)}
+ O_{f}\l( \frac{\log{t}}{\log{\log{t}}} \r)\\
&= \log{\l(\frac{\log{\log{t}}}{\log{X}}\r)} \times \sum_{|t - \gamma| \leq \frac{1}{\log{X}}}1 
+ \sum_{\frac{1}{\log{X}} < |t - \gamma| \leq \frac{1}{\log{\log{t}}}}\log{\l( |t - \gamma|\log{\log{t}} \r)}\\
&\qquad+ O_{f}\l( \frac{\log{t}}{\log{\log{t}}} \r).
\end{align*}
Thus, we obtain formula \eqref{Lam_sum_ub}.
In particular, estimates \eqref{Lam_sum_ub_R_2}, \eqref{Lam_sum_lb_R_2}, \eqref{Lam_sum_ub_I_2} 
are easily obtained by formula \eqref{Lam_sum_ub} and estimate \eqref{BINESZ}.
\end{proof}

\section{\textbf{Proof of Theorem \ref{VSIZD}}}

In this section, we prove Theorem \ref{VSIZD}.
We prepare three lemmas, and the proofs of these lemmas are probably standard for experts in this field, 
and so those proofs are briefly.

\begin{lemma}	\label{CSIZ}
Assume the Riemann Hypothesis and \eqref{GSO}. 
Let $\psi(t)$ be a monotonic function with $3 \leq \psi(t) \leq \sqrt{Y(t)}$.
Then we have
\begin{align*}
\tilde{N}\l(t, \frac{1}{\log{\psi(t)}}\r) \ll M(t) + \frac{\log{t}}{\log{\psi(t)}}
\end{align*}
\end{lemma}

\begin{proof}
For $\s \geq \s_X := \frac{1}{2} + \frac{1}{\log{X}}$, by using the following formula (cf. (2.3) in \cite{SS})
\begin{align}	\label{ISF}
\frac{\zeta'}{\zeta}(s)
= -\sum_{n \leq X^2}\frac{\Lam_{X}'(n)}{n^{s}}
+ O\l( X^{1/2- \s}\l( \l| \sum_{n \leq X^2}\frac{\Lam_{X}'(n)}{n^{\s_X+it}} \r|+\log{t} \r) \r),
\end{align}
we have
\begin{align}	\label{CLIN1}
\frac{\zeta'}{\zeta}\l(\s_X+it\r) 
\ll \l| \sum_{n \leq X^2}\frac{\Lam_{X}'(n)}{n^{\s_X+it}} \r| + \log{t}.
\end{align}
Here, the function $\Lam_{X}'(n)$ is defined by 
\begin{align*}
\Lam_{X}'(n) = \l\{
\begin{array}{cl}
\Lam(n)							&\text{if\; $1 \leq n \leq X$,}\\[2mm]
\Lam(n)\log(X^2/n)/\log{X}	&\text{if\; $X \leq n \leq X^2$,}\\[2mm]
0									&\text{otherwise.}
\end{array}
\r.
\end{align*}
By assuming estimate \eqref{GSO} and using partial summation, the right hand side of \eqref{CLIN1} is
\begin{align*}
\ll M(t)\log{X} + \log{t}
\end{align*}
for $X^2 \leq Y(t)$.
On the other hand, by the following formula
\begin{align*}
\Re\l( \frac{\zeta'}{\zeta}(\s+it) \r) 
= \sum_{|t - \gamma| \leq 1}\frac{\s - 1/2}{(\s - 1/2)^2 + (t - \gamma)^2} + O(\log{t}),
\end{align*}
we have
\begin{align*}
\sum_{|t - \gamma| \leq 1}\frac{1 / \log{X}}{(1/\log{X})^2 + (t - \gamma)^2}
\ll M(t)\log{X} + \log{t}.
\end{align*}
Therefore, we have
\begin{align*}
\sum_{|t - \gamma| \leq 1/\log{X}}1
\ll M(t) + \frac{\log{t}}{\log{X}}
\end{align*}
for $X \leq \sqrt{Y(t)}$. 
Hence by putting $X = \psi(t)$, we obtain this lemma.
\end{proof}

\begin{lemma}	\label{CPERZ}
Assume the Riemann Hypothesis and estimate \eqref{GSO}. 
Let $\psi(t)$ be a monotonic function with $3 \leq \psi(t) \leq \sqrt{Y(t)}$. 
Then we have
\begin{align*}
\log{\zeta\l( \frac{1}{2} + \frac{1}{\log{\psi(t)}} + it \r)} \ll M(t) + \frac{\log{t}}{\log{\psi(t)}}.
\end{align*}
\end{lemma}

\begin{proof}
By the formula \eqref{ISF}, we see that
\begin{align*}
\log{\zeta\l( \s_X + it \r)}
= \sum_{2 \leq n \leq X^2}\frac{\Lam_{X}'(n)}{n^{\s_X+it}\log{n}}
+O\l( \frac{1}{\log{X}}\l( \l| \sum_{n \leq X^2}\frac{\Lam_{X}'(n)}{n^{\s_X+it}} \r|+\log{t} \r) \r).
\end{align*}
By using the partial summation, the above right hand side is
\begin{align*}
\ll M(t) + \frac{\log{t}}{\log{X}}
\end{align*}
for $X \leq \sqrt{Y(t)}$.
Hence by putting $X = \psi(t)$, we obtain this lemma.
\end{proof}

\begin{lemma}	\label{CLF}
Assume the Riemann Hypothesis and estimate \eqref{GSO}. 
Let $\psi(t)$ be a monotonic function with $3 \leq \psi(t) \leq \sqrt{Y(t)}$. 
Then, for $\frac{1}{2} \leq \s \leq \frac{1}{2} + \frac{1}{\log{\psi(t)}}$, $t \geq 14$, we have
\begin{align}	\label{CLFE}
\frac{\zeta'}{\zeta}(s) 
= \sum_{|t - \gamma| \leq \frac{1}{\log{\psi(t)}}}\frac{1}{s - \rho} 
+ O(M(t)\log{\psi(t)} + \log{t}).
\end{align}
\end{lemma}

\begin{proof}
We can obtain this lemma by using Lemma \ref{CSIZ} and the same method as in the proof of Lemma 13.20 in \cite{MV}.
\end{proof}

\begin{proof}[Proof of Theorem \ref{VSIZD}]
Let $\psi(t) \leq X \leq t$.
Using \eqref{BLDZ}, Lemma \ref{CSIZ}, and Lemma \ref{CLF}, we can find that
\begin{align*}
\sum_{|t - \gamma| > \frac{1}{\log{\psi(t)}}}\frac{1}{(t - \gamma)^2}
\ll \log{\psi(t)}(M(t)\log{\psi(t)} + \log{t}).
\end{align*}
Therefore, by using this estimate and Theorem \ref{Main_Prop}, we have
\begin{multline}	\label{JRHF}
\sum_{2 \leq n \leq X^2}\frac{\Lam(n)v_{f, 1}(e^{\log{n}/\log{X}})}{n^{1/2+it}\log{n}}\\
= \log{\zeta\l( \frac{1}{2} + it \r)}
-\sum_{|t - \gamma| \leq \frac{1}{\log{X}}}\log(|t - \gamma|\log{X}) 
+ O\l(M(t) + \frac{\log{t}}{\log{\psi(t)}}\r).
\end{multline}
On the other hand, by integrating the both sides of \eqref{CLFE}, we find that
\begin{multline*}
\log{\zeta\l( \frac{1}{2} + it \r)} - \log{\zeta\l( \frac{1}{2} + \frac{1}{\log{\psi(t)}} + it \r)}\\
= \sum_{|t - \gamma| \leq \frac{2}{\log{Y(t)}}}\log\l(\frac{i(t - \gamma)}{\frac{1}{\log{\psi(t)}} + i(t - \gamma)}\r) 
+ O\l(M(t) + \frac{\log{t}}{\log{\psi(t)}} \r).
\end{multline*}
Hence, using Lemma \ref{CSIZ} and Lemma \ref{CPERZ}, we have
\begin{align*}
\log{\zeta\l(\frac{1}{2}+it\r)}
= \sum_{|t - \gamma| \leq \frac{1}{\log{\psi(t)}}}\log\l(|t  - \gamma|\log{\psi(t)}\r) 
+ O\l(M(t) + \frac{\log{t}}{\log{\psi(t)}} \r).
\end{align*}
By this formula, the right hand side of \eqref{JRHF} is equal to
\begin{multline*}
\log\l( \frac{\log{\psi(t)}}{\log{X}} \r) \times \tilde{N}\l( t, \frac{1}{\log{X}} \r)
+\sum_{\frac{1}{\log{X}} < |t - \gamma| \leq \frac{1}{\log{\psi(t)}}}\log\l(|t - \gamma|\log{\psi(t)} \r)\\ 
+ O\l(M(t) + \frac{\log{t}}{\log{\psi(t)}}\r).
\end{multline*}
On the other hand, we see that the left hand side of \eqref{JRHF} is $= P_{f}(1/2+it) + O(\log{\log{X}})$,
which completes the proof of Theorem \ref{VSIZD}.
\end{proof}


\section{\textbf{Proof of Theorem \ref{LVEJ}}}

In this section, we prove Theorem \ref{LVEJ}.
We will use the method of Selberg-Tsang \cite{KTDT} in a part of the proof, 
where the following proposition plays an important role there.
Moreover, the proposition also plays an important role in the proof of Theorem \ref{QvM}.

Before stating the proposition, we define $\s_{X, t}$ and $\Lam_{X}(n) = \Lam(n)w_{X}(n)$ by 
\begin{gather}
\label{def_s_X}
\s_{X, t}
= \frac{1}{2} + 2\max_{|t - \gamma| \leq \frac{X^{3(\b - 1/2)}}{\log{X}}}\l\{ \b - \frac{1}{2}, \frac{2}{\log{X}} \r\},\\
\label{def_w_X}
w_{X}(y) = \l\{
\begin{array}{cl}
1																	&\text{if\, $1 \leq y \leq X$,}\\[2mm]
\frac{(\log(X^3/y))^2 - 2(\log(X^2 / y))^2}{2(\log{X})^2}		&\text{if\, $X \leq y \leq X^2$,}\\[2mm]
\frac{(\log(X^{3}/y))^2}{2(\log{X})^2}							&\text{if\, $X^{2} \leq y \leq X^3$.}
\end{array}
\r.
\end{gather}
Then, we can obtain the following proposition.

\begin{proposition}	\label{RCSFP}
Assume $D(f) \geq 2$.
Let $m$ be a nonnegative integer, and let $X$, $H$ be real parameters with $X \geq 3$, $H \geq 1$. 
Then, for $t \geq 14$, $\s \geq 1/2$, the right hand side of \eqref{ESRm} is estimated by
\begin{multline}
\ll_{f} \frac{X^{2(1-\s)} + X^{1-\s}}{t(\log{X})^{m+1}} +\\
+H^{3}\frac{\s_{X, t} - 1/2}{(\log{X})^{m}}(X^{2(\s_{X, t} - \s)} + X^{\s_{X, t} - \s})
\l( \Bigg| \sum_{n \leq X^3}\frac{\Lam_{X}(n)}{n^{\s_{X, t} + it}} \Bigg| + \log{t} \r).
\end{multline}
\end{proposition}

Thanks to Proposition \ref{RCSFP}, we can combine the method of Selberg-Tsang with Theorem \ref{Main_Prop}.

\begin{proof}
By estimate \eqref{ESRm} and the line symmetry of nontrivial zeros of $\zeta(s)$ with respect to $\s=1/2$, 
it suffices to show that
\begin{multline}	\label{pkRCSFP}
\us{\b \geq 1/2}{\sum_{|t - \gamma| \leq \frac{1}{\log{X}}}}(X^{2(\b - \s)} + X^{\b - \s})
+ \frac{1}{(\log{X})^3}\us{\b \geq 1/2}{\sum_{|t - \gamma| > \frac{1}{\log{X}}}}\frac{X^{2(\b - \s)} + X^{\b - \s}}{|t - \gamma|^3}\\
\ll (\s_{X, t} - 1/2)(X^{2(\s_{X, t} - \s)} + X^{\s_{X, t} - \s})
\l( \Bigg| \sum_{n \leq X^3}\frac{\Lam_{X}(n)}{n^{\s_{X, t} + it}} \Bigg| + \log{t} \r).
\end{multline}
If $\b > \frac{\s_{X, t} + 1/2}{2}$, then by the definition of $\s_{X, t}$ \eqref{def_s_X}, we have
\begin{align*}
|t - \gamma| > \frac{X^{3(\b - 1/2)}}{\log{X}} > 3(\b - 1/2) > 3|\s_{X, t} - \b|.
\end{align*}
By these inequalities, we find that
\begin{multline*}
\frac{X^{2(\b - \s)} + X^{\b - \s}}{|t - \gamma|^3}
\ll \frac{\log{X}}{X^{3(\b - 1/2)}}\frac{X^{2(\b - \s)} + X^{\b - \s}}{(\s_{X, t} - \b)^2 + (t - \gamma)^2}\\
\ll X^{1/2 - \s}(\log{X})^2\frac{\s_{X, t} - 1/2}{(\s_{X, t} - \b)^2 + (t - \gamma)^2}.
\end{multline*}
Next, we suppose $1/2 \leq \b \leq \frac{\s_{X, t} + 1/2}{2}$. 
Then if $|t - \gamma| > \s_{X, t} - 1/2$, we find that
\begin{align*}
\frac{X^{2(\b - \s)} + X^{\b - \s}}{|t - \gamma|^3}
\ll (X^{2(\s_{X, t} - \s)} + X^{\s_{X, t} - \s})(\log{X})^2\frac{\s_{X, t} - 1/2}{(\s_{X, t} - \b)^2 + (t - \gamma)^2},
\end{align*}
and if $1/\log{X} < |t - \gamma| \leq \s_{X, t} - 1/2$, we find that
\begin{align*}
\frac{X^{2(\b - \s)} + X^{\b - \s}}{|t - \gamma|^3}
\ll (X^{2(\s_{X, t} - \s)} + X^{\s_{X, t} - \s})(\log{X})^3 
\frac{(\s_{X, t} - 1/2)^2}{(\s_{X, t} - \b)^2 + (t - \gamma)^2}.
\end{align*}
From the above estimates, we have
\begin{multline}	\label{RCSFPE1}
\frac{1}{(\log{X})^3}\us{\b \geq 1/2}{\sum_{|t - \gamma| > \frac{1}{\log{X}}}}
\frac{X^{2(\b - \s)} + X^{\b - \s}}{|t - \gamma|^3}\\
\ll (\s_{X, t} - 1/2)(X^{2(\s_{X, t} - \s)} + X^{\s_{X, t} - \s})
\sum_{|t - \gamma| > \frac{1}{\log{X}}}\frac{\s_{X, t} - 1/2}{(\s_{X, t} - \b)^2 + (t - \gamma)^2}.
\end{multline}
Moreover, it holds that
\begin{multline*}
\us{\b \geq 1/2}{\sum_{|t - \gamma| \leq \frac{1}{\log{X}}}}(X^{2(\b - \s)} + X^{\b - \s})\\
\ll (\s_{X, t} - 1/2)(X^{2(\s_{X, t} - \s)} + X^{\s_{X, t} - \s})\sum_{|t - \gamma| \leq \frac{1}{\log{X}}}
\frac{\s_{X, t} - 1/2}{(\s_{X, t} - \b)^2 + (t - \gamma)^2}.
\end{multline*}
By this estimate and \eqref{RCSFPE1}, we obtain
\begin{multline*}
\us{\b \geq 1/2}{\sum_{|t - \gamma| \leq \frac{1}{\log{X}}}}(X^{2(\b - \s)} + X^{\b - \s})
+ \frac{1}{(\log{X})^3}\us{\b \geq 1/2}{\sum_{|t - \gamma| > \frac{1}{\log{X}}}}
\frac{X^{2(\b - \s)} + X^{\b - \s}}{|t - \gamma|^3}\\
\ll (\s_{X, t} - 1/2)(X^{2(\s_{X, t} - \s)} + X^{\s_{X, t} - \s})\sum_{\rho}
\frac{\s_{X, t} - 1/2}{(\s_{X, t} - \b)^2 + (t - \gamma)^2}.
\end{multline*}

Here, we have the following estimates (cf.\ (4.4) and (4.9) in \cite{SCR})
\begin{align}	\label{BLFR}
\sum_{\rho}\frac{\s_{X, t} - 1/2}{(\s_{X, t} - \b)^2 + (t - \gamma)^2} 
\ll \l| \sum_{n \leq X^3}\frac{\Lam_{X}(n)}{n^{\s_{X, t}+it}} \r| + \log{t}.
\end{align}
Thus, we obtain this proposition.
\end{proof}

Moreover, we prepare some lemmas.

\begin{lemma}
\label{SLL}
Let $T \geq 5$, and let $3 \leq X \leq T$. Let $k$ be a positive integer such that $X^{k} \leq T / \log{T}$.
Then, for any complex numbers $a(p)$, we have
\begin{align*}
\int_{0}^{T}\bigg| \sum_{p \leq X}\frac{a(p)}{p^{1/2 + it}} \bigg|^{2k}dt
\ll T k! \l( \sum_{p \leq X}\frac{|a(p)|^2}{p} \r)^{k}.
\end{align*}
Here, the above sums run over prime numbers.
\end{lemma}

\begin{proof}
This lemma is a little modified assertion of Lemma 3 in \cite{SM2009}, and the proof of this lemma is the same as its proof.
\end{proof}

\begin{lemma}	\label{SL_s_X}
Let $T \geq 5$, and let $k$ be a positive integer, 
$X \geq 3$, $\xi \geq 1$ be some parameters with $X^{15}\xi^{10} \leq T$.
Then, we have
\begin{align*}
\int_{0}^{T}\l( \s_{X, t} - \frac{1}{2} \r)^{k}\xi^{\s_{X, t} - 1/2}dt
\ll T \l( \frac{4^{k} \xi^{\frac{4}{\log{X}}}}{(\log{X})^{k}} + \frac{8^{k} k!}{\log{X}(\log{T})^{k-1}} \r).
\end{align*}
\end{lemma}

\begin{proof}
This lemma is a little modified assertion of Lemma 12 in \cite{SCR}, 
and the proof of this lemma is the same as its proof.
\end{proof}

\begin{lemma}	\label{RPLD1}
Let $T$ be large, $X = T^{1/(\log{\log{T}})^2}$. 
Then, for $1 \ll V = o(\sqrt{\log{\log{T}}})$, we have
\begin{align*}
\frac{1}{T}\meas
\set{t \in [T, 2T]}{\Re\sum_{p \leq X}\frac{1}{p^{\frac{1}{2}+it}} > V\sqrt{\frac{1}{2}\sum_{p \leq X}\frac{1}{p}}}
= (1 + o(1))\int_{V}^{\infty}e^{\frac{-u^2}{2}}\frac{du}{\sqrt{2\pi}}.
\end{align*}
\end{lemma}

\begin{proof}
This lemma is Proposition 1 in \cite{Ra2011}.
\end{proof}

\begin{proof}[Proof of Theorem \ref{LVEJ}]
Let $T$ be large, and $V$ a parameter with $\sqrt{\log{\log{T}}} \ll V \ll (\log{\log{T}})^{2/3}$. 
Here, we may assume the inequality $V \leq A(\log{\log{T}})^{2/3}$ with $A$ any fixed positive constant.
Then, it suffices to show  that, as $T \rightarrow +\infty$
\begin{multline*}
\frac{1}{T}\meas(\S(T, V))\\
\leq (1 + o(1))\int_{\frac{V}{\sqrt{1/2\log{\log{T}}}}}^{\infty}e^{-u^2/2}\frac{du}{\sqrt{2\pi}}
+ O\l( \frac{V}{(\log{\log{T}})^{5/6}}\exp\l( -\frac{V^2}{\log{\log{T}}} \r) \r).
\end{multline*}
Let $X$, $Y$ be parameters with $X = T^{1/(\log{\log{T}})^2} \leq Y \leq T^{1/100}$.
Let $f$ be a fixed function satisfying the condition of this paper and $D(f) \geq 2$. 
By Theorem \ref{Main_Prop} and Proposition \ref{RCSFP}, for $T \leq t \leq 2T$, we have
\begin{multline}
\log{|\zeta(1/2+it)|}
\leq \Re\sum_{2 \leq n \leq Y^2}\frac{\Lam(n)v_{f, 1}(e^{\log{n} / \log{Y}})}{n^{1/2+it} \log{n}}\\
+C_1(\s_{Y, t} - 1/2)Y^{2\s_{Y, t} - 1}
\l( \Bigg| \sum_{n \leq Y^3}\frac{\Lam_{Y}(n)}{n^{\s_{Y, t} + it}} \Bigg| + \log{T} \r),
\end{multline}
where $C_1$ is an absolute positive constant.
Now, we see that
\begin{alignat*}{2}
&\Re\sum_{2 \leq n \leq Y^2}\frac{\Lam(n)v_{f, 1}(e^{\log{n} / \log{Y}})}{n^{1/2+it} \log{n}}\\
&= \Re\sum_{p \leq X}\frac{1}{p^{1/2+it}} + \Re\sum_{X < p \leq Y^2}\frac{v_{f, 1}(e^{\log{p} / \log{Y}})}{p^{1/2+it}}&
&+ \Re\sum_{p \leq Y}\frac{v_{f, 1}(e^{\log{p^2}/\log{Y}})}{p^{1+2it} \log{p^2}}\\
& & &+ \Re\us{k \geq 3}{\sum_{p^{k} \leq Y^2}}\frac{\Lam(p^{k})v_{f, 1}(e^{\log{p^{k}}/\log{Y}})}{p^{k(1/2+it)} \log{p^{k}}},
\end{alignat*}
\begin{align*}
\bigg|\us{k \geq 3}{\sum_{p^{k} \leq Y^2}}\frac{\Lam(p^{k})v_{f, 1}(e^{\log{p^{k}}/\log{Y}})}{p^{k(1/2+it)} \log{p^{k}}}\bigg|
\leq \us{k \geq 3}{\sum_{p^{k} \leq Y^2}}\frac{\Lam(p^{k})}{p^{k/2}\log{p^{k}}}
\ll 1, 
\end{align*}
and that
\begin{align*}
\bigg|\us{k \geq 2}{\sum_{p^{k} \leq Y^3}}\frac{\Lam_{Y}(p^{k})}{p^{k(\s_{Y, t} + it)}}\bigg|
\leq \us{k \geq 2}{\sum_{p^{k} \leq Y^3}}\frac{\log{p}}{p^{k\s_{Y, t}}}
\leq \log{Y} + O(1) \leq \log{T}.
\end{align*}
Hence, we have
\begin{align}	\label{TINES}
\meas(\S(T, V))
\leq \meas(S_1) + \meas(S_2) + \meas(S_3) + \meas(S_4), 
\end{align}
where the sets $S_1$, $S_2$, $S_{3}$, $S_{4}$ are defined by
\begin{gather*}
S_1 := \l\{ t \in [T, 2T] \; \Bigg| \; 
\Re \sum_{p \leq X}\frac{1}{p^{1/2+it}} > V_1 \r\}, \\
S_2 : = \set{t \in [T, 2T]}{\Re\sum_{X < p \leq Y^2}\frac{v_{f, 1}(e^{\log{p}/\log{Y}})}{p^{1/2+it}} > V_2},\\
S_3 := \l\{ t \in [T, 2T] \; \Bigg| \; 
\Re \sum_{p \leq Y}\frac{v_{f, 1}(e^{\log{p^2}/\log{Y}})}{p^{1+2it}} > V_2 \r\}, \\
S_4 := \l\{ t \in [T, 2T] \; \Bigg| \; 
C_1(\s_{Y, t} - 1/2)Y^{2\s_{Y, t} - 1}
\l( \Bigg| \sum_{p \leq Y^3}\frac{\Lam_{Y}(p)}{p^{\s_{Y, t} + it}} \Bigg| + 2\log{T} \r) > V_2 \r\},
\end{gather*}
where $V_{1} = V - 3V_2$, and $V_2$ is a positive parameter with $V_2 \leq V/4$.
Let $k$ be a positive integer with $k \leq  \frac{1}{100}\frac{\log{T}}{\log{Y}}$.
By Lemma \ref{RPLD1}, we find that
\begin{align}	\label{INE0_EJ}
\int_{T}^{2T}\bigg| \sum_{X < p \leq Y^2}\frac{v_{f, 1}(e^{\log{p}/\log{Y}})}{p^{1/2+it}} \bigg|^{2k}dt
\ll T \l( C_2 k\log{\log{\log{T}}} \r)^{k},
\end{align}
and that
\begin{align}	\label{INE1_EJ}
\int_{T}^{2T}\bigg| \sum_{p \leq X}\frac{v_{f, 1}(e^{\log{p^2}/\log{X}})}{p^{1+2it}} \bigg|^{2k}dt
\ll T k! C_{3}^{k}.
\end{align}
By Lemma \ref{SL_s_X}, we have
\begin{align}	\label{INE2_EJ}
\int_{T}^{2T} (2C_1)^{k}(\s_{Y, t} - 1/2)^{k}Y^{2k(\s_{Y, t} - 1/2)}(\log{T})^{k}dt
\ll T \l( \frac{C_3 \log{T}}{\log{Y}} \r)^{k}.
\end{align}

Now, we can write
\begin{align*}
\sum_{p \leq Y^3}\frac{\Lam_{Y}(p)}{p^{\s_{Y, t} + it}}
&= \sum_{p \leq Y^3}\frac{\Lam_{Y}(p)}{p^{1/2 + it}} 
- \sum_{p \leq Y^3}\frac{\Lam_{Y}(p)}{p^{1/2 + it}}(1 - p^{1/2 - \s_{Y, t}})\\
&= \sum_{p \leq Y^3}\frac{\Lam_{Y}(p)}{p^{1/2 + it}}
-\int_{1/2}^{\s_{Y, t}}\sum_{p \leq Y^3}\frac{\Lam_{Y}(p)\log{p}}{p^{\a' + it}}d\a',
\end{align*}
and, for $1/2 \leq \a' \leq \s_{Y, t}$, 
\begin{align*}
\bigg| \sum_{p \leq Y^3}\frac{\Lam_{Y}(p)\log{p}}{p^{\a' + it}} \bigg|
&= Y^{\a' - 1/2}\Bigg|\int_{\a'}^{\infty}Y^{1/2 - \a}
\sum_{p \leq Y^3}\frac{\Lam_{Y}(p) \log{(Y p)} \log{p}}{p^{\a + it}}d\a\Bigg|\\
&\leq Y^{\s_{Y, t}  - 1/2}\int_{1/2}^{\infty}Y^{1/2 - \a}
\bigg|\sum_{p \leq Y^3}\frac{\Lam_{Y}(p) \log{(Y p)} \log{p}}{p^{\a + it}}\bigg|d\a.
\end{align*}
Therefore, we have
\begin{multline}	\label{KINs_X}
\bigg| \sum_{p \leq Y^3}\frac{\Lam_{Y}(p)}{p^{\s_{Y, t} + it}} \bigg|
\leq \bigg| \sum_{p \leq Y^3}\frac{\Lam_{Y}(p)}{p^{1/2 + it}} \bigg| +\\
+ (\s_{Y, t} - 1/2)Y^{\s_{Y, t}  - 1/2}\int_{1/2}^{\infty}Y^{1/2 - \a}
\bigg|\sum_{p \leq Y^3}\frac{\Lam_{Y}(p) \log{(Y p)} \log{p}}{p^{\a + it}}\bigg|d\a.
\end{multline}

By the Cauchy-Schwarz inequality, and Lemmas \ref{SLL}, \ref{SL_s_X}, we have
\begin{align}	\label{INE3_EJ}
&\int_{T}^{2T}(\s_{Y, t} - 1/2)^{k}Y^{2k(\s_{Y, t} - 1/2)}
\bigg| \sum_{p \leq Y^3}\frac{\Lam_{Y}(p)}{p^{1/2 + it}} \bigg|^{k}dt\\ \nonumber
&\leq \Bigg(\int_{T}^{2T}(\s_{Y, t} - 1/2)^{2k}Y^{4k(\s_{Y, t} - 1/2)}dt\Bigg)^{1/2}
\Bigg(\int_{T}^{2T}\bigg| \sum_{p \leq Y^3}\frac{\Lam_{Y}(p)}{p^{1/2 + it}} \bigg|^{2k}dt\Bigg)^{1/2}\\ \nonumber
&\ll T(C k^{1/2})^{k}.
\end{align}
On the other hand, by the Cauchy-Schwarz inequality and Lemma \ref{SL_s_X}, we find that
\begin{align*}
&\int_{T}^{2T}(\s_{Y, t}-1/2)^{2k}Y^{3k(\s_{Y, t} - 1/2)}
\l(\int_{1/2}^{\infty}Y^{1/2 - \a}
\bigg|\sum_{p \leq Y^3}\frac{\Lam_{Y}(p) \log{(Y p)} \log{p}}{p^{\a + it}}\bigg|d\a\r)^{k}dt\\
&\leq \Bigg(\int_{T}^{2T}(\s_{Y, t}-1/2)^{4k}Y^{6k(\s_{Y, t} - 1/2)}dt\Bigg)^{1/2} \times\\
&\qqquad \times \l(\int_{T}^{2T}\l( \int_{1/2}^{\infty}Y^{1/2 - \a} 
\bigg|\sum_{p \leq Y^3}\frac{\Lam_{Y}(p) \log{(Y p)} \log{p}}{p^{\a + it}}\bigg|d\a\r)^{2k}dt \r)^{1/2}\\
&\ll \frac{T^{1/2} C^{k}}{(\log{Y})^{2k}}
\l(\int_{T}^{2T}\l( \int_{1/2}^{\infty}Y^{1/2 - \a} 
\bigg|\sum_{p \leq Y^3}\frac{\Lam_{Y}(p) \log{(Y p)} \log{p}}{p^{\a + it}}\bigg|d\a\r)^{2k}dt \r)^{1/2}.
\end{align*}
Moreover, by H\"older's inequality, we have
\begin{align*}
&\l( \int_{1/2}^{\infty}Y^{1/2 - \a} 
\bigg|\sum_{p \leq Y^3}\frac{\Lam_{Y}(p) \log{(Y p)} \log{p}}{p^{\a + it}}\bigg|d\a\r)^{2k}\\
&\leq \l(\int_{1/2}^{\infty}Y^{1/2 - \a}d\a\r)^{2k-1} \times
\l( \int_{1/2}^{\infty} Y^{1/2 - \a}\bigg|\sum_{p \leq Y^3}\frac{\Lam_{Y}(p) \log{(Y p)} \log{p}}{p^{\a + it}}\bigg|^{2k}d\a \r)\\
&= \frac{1}{(\log{Y})^{2k-1}}
\int_{1/2}^{\infty} Y^{1/2 - \a}\bigg|\sum_{p \leq Y^3}\frac{\Lam_{Y}(p) \log{(Y p)} \log{p}}{p^{\a + it}}\bigg|^{2k}d\a.
\end{align*}
Therefore, by using Lemma \ref{SLL}, we find that
\begin{align*}
&\int_{T}^{2T}\l( \int_{1/2}^{\infty}Y^{1/2 - \a} 
\bigg|\sum_{p \leq Y^3}\frac{\Lam_{Y}(p) \log{(Y p)} \log{p}}{p^{\a + it}}\bigg|d\a\r)^{2k}dt\\
&\leq \frac{1}{(\log{Y})^{2k-1}}\int_{1/2}^{\infty}Y^{1/2-\a}\l(
\int_{T}^{2T}\bigg|\sum_{p \leq Y^3}\frac{\Lam_{Y}(p) \log{(Y p)} \log{p}}{p^{\a + it}}\bigg|^{2k}dt\r)d\a\\
&\ll \frac{T k!}{(\log{Y})^{2k-1}}\int_{1/2}^{\infty}Y^{1/2-\a}
\l( \sum_{p \leq Y^3}\frac{(\log(Y p))^2(\log{p})^4}{p^{2\a}} \r)^{k}d\a\\
&\ll T k!C^{k} (\log{Y})^{4k+1} \int_{1/2}^{\infty}Y^{1/2-\a}d\a 
\leq T k!C^{k} (\log{Y})^{4k}.
\end{align*}
Hence, we obtain
\begin{align*}
\int_{T}^{2T}(\s_{Y, t}-1/2)^{2k}Y^{3k(\s_{Y, t} - 1/2)}
\l(\int_{1/2}^{\infty}Y^{1/2 - \a}
\bigg|\sum_{p \leq X^3}\frac{\Lam_{Y}(p) \log{(Y p)} \log{p}}{p^{\a + it}}\bigg|d\a\r)^{k}dt\\
\ll T (C k^{1/2})^{k}.
\end{align*}
By this estimate and estimates \eqref{INE2_EJ}, \eqref{KINs_X}, \eqref{INE3_EJ}, we have
\begin{multline}	\label{KINEs_X}
\frac{1}{T}\int_{T}^{2T}C_{1}^{k}(\s_{Y, t} - 1/2)^{k}Y^{2k(\s_{Y, t} - 1/2)}
\l( \Bigg| \sum_{p \leq Y^3}\frac{\Lam_{Y}(p)}{p^{\s_{Y, t} + it}} \Bigg| + 2\log{T} \r)^{k}dt\\
\ll \l( \frac{C_4 \log{T}}{\log{Y}} \r)^{k}.
\end{multline}
Thus, by estimates \eqref{INE0_EJ}, \eqref{INE1_EJ}, \eqref{KINEs_X}, the following estimates
\begin{align}	\label{KESS}
\begin{gathered}
\frac{1}{T}\meas(S_2) \ll \l( \frac{k C_2 \log{\log{\log{T}}}}{V_{2}^{2}} \r)^{k}, \\
\frac{1}{T}\meas(S_3) \ll \l( \frac{k C_3}{V_{2}^{2}} \r)^{k}, \quad
\frac{1}{T}\meas(S_4) \ll \l( \frac{C_4 \log{T}}{V_2 \log{Y}} \r)^{k}
\end{gathered}
\end{align}
hold for $X \leq Y \leq T^{1/100}$, $k \leq \frac{1}{100}\frac{\log{T}}{\log{Y}}$.

We put $Y = T^{\log{\log{T}} / (200C_5V^2)}$ and $k = 2\l[\frac{V^2}{\log{\log{T}}} + 1\r]$, 
where $C_5$ is a constant chosen as satisfying $C_5 \geq 2$ and $C_5V^2/\log{\log{T}} \geq 2$.
Further, we decide $V_2$ as $200C_4 C_5 e^2 A V / (\log{\log{T}})^{1/3}$
Then we obtain
\begin{align*}
\frac{\meas(S_2) + \meas(S_3) + \meas(S_4)}{T}
\ll \exp\l( -\frac{2V^2}{\log{\log{T}}}\log\l( \frac{e A(\log{\log{T}})^{2/3}}{V} \r)\r)
\end{align*}
for $\sqrt{\log{\log{T}}} \ll V \leq A(\log{\log{T}})^{2/3}$.
Hence, by Lemma \ref{RPLD1} and inequality \eqref{TINES}, we have
\begin{align*}
\frac{1}{T}\meas(\S(T, V))
\leq (1 + o(1))\int_{\frac{V_1}{W(T)}}^{\infty}e^{-u^2/2}\frac{du}{\sqrt{2\pi}}
+ o\l( \int_{\frac{V}{\sqrt{1/2\log{\log{T}}}}}^{\infty}e^{-u^2/2}du \r)
\end{align*}
for $\sqrt{\log{\log{T}}} \ll V \leq A(\log{\log{T}})^{2/3}$.
Here, $W(T)$ indicates
\begin{align*}
W(T) = \sqrt{\frac{1}{2}\sum_{p \leq X}p^{-1}} 
= \sqrt{\frac{1}{2}\log{\log{T}}} + O\l( \frac{\log{\log{\log{T}}}}{\sqrt{\log{\log{T}}}} \r).
\end{align*}
Here, we find that
\begin{align*}
\int_{\frac{V_1}{W(T)}}^{\infty}e^{-u^2/2}\frac{du}{\sqrt{2\pi}}
= \int_{\frac{V}{\sqrt{1/2\log{\log{T}}}}}^{\infty}e^{-u^2/2}\frac{du}{\sqrt{2\pi}}
+ \int_{\frac{V_1}{W(T)}}^{\frac{V}{\sqrt{1/2\log{\log{T}}}}}e^{-u^2/2}\frac{du}{\sqrt{2\pi}}
\end{align*}
and that
\begin{align*}
\int_{\frac{V_1}{W(T)}}^{\frac{V}{\sqrt{1/2\log{\log{T}}}}}e^{-u^2/2}\frac{du}{\sqrt{2\pi}}
&\ll \l( \frac{V}{\sqrt{1/2\log{\log{T}}}} - \frac{V_1}{W(T)} \r)e^{-\frac{V_{1}^2}{2W(T)^2}}\\
&\ll \frac{V}{(\log{\log{T}})^{5/6}}e^{-\frac{V^2}{\log{\log{T}}}}.
\end{align*}
Thus, we have
\begin{multline*}
\frac{1}{T}\meas(\S(T, V))\\
\leq (1 + o(1))\int_{\frac{V}{\sqrt{1/2\log{\log{T}}}}}^{\infty}e^{-u^2/2}\frac{du}{\sqrt{2\pi}}
+ O\l( e^{-\frac{V^2}{\log{\log{T}}}}\frac{V}{(\log{\log{T}})^{5/6}} \r)
\end{multline*}
for $\sqrt{\log{\log{T}}} \ll V \leq A(\log{\log{T}})^{2/3}$. 
This completes the proof of Theorem \ref{LVEJ}.
\end{proof}

\section{\textbf{Proofs of Theorem \ref{QvM} and Theorem \ref{MLFmGM}}}

In this section, we prove Theorem \ref{QvM} and Theorem \ref{MLFmGM}.

\begin{proof}[Proof of Theorem \ref{QvM}]
Let $m$ be a positive integer and $f$ be a fixed function satisfying the condition of this paper and $D(f) \geq 2$.
Then, by Theorem \ref{Main_Prop}, for $t \geq 14$, $X \leq T^{\frac{1}{135k}} = : Y$, we obtain
\begin{multline}	\label{QvM1}
\Bigg|\eta_{m}(\s + it) - i^{m}\sum_{2 \leq n \leq X}\frac{\Lam(n)}{n^{\s + it}(\log{n})^{m+1}} - Y_{m}(\s+it)\Bigg|^{2k}\\
\leq 2^{2k}\Bigg|\sum_{X < n \leq Y^2}\frac{\Lam(n)v_{f, 1}(e^{\log{n}/\log{Y}})}{n^{\s + it}(\log{n})^{m+1}}\Bigg|^{2k}
+ 2^{2k}|R_{m}(\s + it, Y, 1)|^{2k}.
\end{multline}
By using partial summation, Lemma \ref{SLL}, and the prime number theorem, we find that
\begin{align*}
\int_{0}^{T}\Bigg| \sum_{X < p \leq Y^2}\frac{v_{f, 1}(e^{\log{p}/\log{Y}})}{p^{\s + it}(\log{p})^{m}} \Bigg|^{2k}dt
&\ll T k!\Bigg(\sum_{p > X}\frac{1}{p^{2\s}(\log{p})^{2m}}\Bigg)^{k}\\
&\leq T k! \l( \frac{2m+1}{2m} + \frac{C}{\log{X}} \r)^{k}\frac{X^{k(1-2\s)}}{(\log{X})^{2k m}},
\end{align*}
and that
\begin{align*}
\int_{0}^{T}\Bigg| \sum_{X < p^2 \leq Y^2}\frac{v_{f, 1}(e^{\log{p^2}/\log{Y}})}{p^{2\s + 2it}(\log{p^2})^{m}} \Bigg|^{2k}dt
&\ll T k!\Bigg(\sum_{p > \sqrt{X}}\frac{1}{p^{4\s}(\log{p^2})^{2m}}\Bigg)^{k}\\
&\leq T k! C^{k}\frac{X^{k(1-4\s)/2}}{(\log{X})^{2k m}}.
\end{align*}
Set 
\begin{gather*}
\psi_{3}(z, y) := \us{l \geq 3}{\sum_{y < p^{l} \leq z}}\log{p}.
\end{gather*}
Then we can easily obtain the inequality $\psi_3(z, y) \ll z^{1/3}$.
By using this inequality and partial summation, we find that
\begin{align*}
\Bigg|\us{l \geq 3}{\sum_{X < p^{l} \leq Y^2}}\frac{v_{f, 1}(e^{\log{p^{l}}/\log{Y}})}{lp^{l(\s + it)}(\log{p^{l}})^{m}}\Bigg|
\leq \int_{X}^{\infty}\frac{\s \log{\xi} + m}{\xi^{1 + \s}(\log{\xi})^{m+1}}\psi_{3}(\xi, X)d\xi
\ll \frac{X^{1/3 - \s}}{(\log{X})^{m}}.
\end{align*}
Therefore, we have
\begin{align*}
\int_{0}^{T}\Bigg|\us{l \geq 3}{\sum_{X < p^{l} \leq Y^2}}
\frac{v_{f, 1}(e^{\log{p^{l}}/\log{Y}})}{lp^{l(\s + it)}(\log{p^{l}})^{m}}\Bigg|^{2k}dt
\ll T C^{k}\frac{X^{k(2/3 - 2\s)}}{(\log{X})^{2 k m}}.
\end{align*}
Hence it holds that
\begin{multline}	\label{QvM2}
\int_{0}^{T}\Bigg| \sum_{X < n \leq Y^2}\frac{\Lam(n)v_{f, 1}(e^{\log{n}/\log{Y}})}{n^{1/2+it}(\log{n})^{m+1}} \Bigg|^{2k}dt\\
\ll T k! \l( \frac{2m+1}{2m} + \frac{C}{\log{X}} \r)^{k} \frac{X^{k(1 - 2\s)}}{(\log{X})^{2k m}}.
\end{multline}

Next, we consider the integral of $R_{m}(s, Y, 1)$. By Proposition \ref{RCSFP}, we have
\begin{align*}
&\int_{14}^{T}|R_{m}(\s+it, Y, 1)|^{2k}dt
\ll (C k^{2(m+1)})^{k} \frac{T^{1-\s} + T^{(1-\s)/2}}{(\log{T})^{2k(m+1)}}+ \\
&+\frac{(C k^{2m})^{k}Y^{(1 - 2\s)k}}{(\log{T})^{2km}}
\int_{14}^{T}\Bigg\{\l(\s_{Y, t} - \frac{1}{2}\r)Y^{2\s_{Y, t} - 1}
\bigg( \bigg| \sum_{n \leq Y^3}\frac{\Lam_{Y}(n)}{n^{\s_{Y, t} + it}} \bigg| + \log{t} \bigg)\Bigg\}^{2k}dt,
\end{align*}
where $\Lam_{Y}(n) = \Lam(n)w_{Y}(n)$, and $w_{Y}(n)$ is given by \eqref{def_w_X}.
By the same method as the proof of estimate \eqref{KINEs_X}, we can obtain
\begin{multline*}
\int_{0}^{T}\Bigg\{\l(\s_{Y, t} - \frac{1}{2}\r)Y^{2\s_{Y, t} - 1}
\bigg( \bigg| \sum_{n \leq Y^3}\frac{\Lam_{Y}(n)}{n^{\s_{Y, t} + it}} \bigg| + \log{(t+2)} \bigg)\Bigg\}^{2k}dt
\ll T (C k^{2})^{k}.
\end{multline*}
Hence, we have
\begin{align}	\label{QvM3}
\int_{14}^{T}|R_{m}(\s+it, Y, 1)|^{2k}dt
\ll T^{1 + \frac{1 - 2\s}{135}}\frac{C^{k} k^{2k(m+1)}}{(\log{T})^{2km}}.
\end{align}

Thus, from this estimate, \eqref{QvM1}, and \eqref{QvM2}, we obtain Theorem \ref{QvM}.
\end{proof}

\begin{proof}[Proof of Theorem \ref{MLFmGM}]
Let $m$ be a positive integer.
Let $X$, $T$ be sufficiently large numbers with $X \leq T^{\frac{1}{135k}}$.
Set $V$ be any positive number.
By Theorem \ref{QvM}, there exists a positive number $C_1 > 3$ such that
\begin{align}	\label{PJIN1}
\meas(\T_{m}(T, X, V))
\ll \sqrt{k}\l(\frac{4k(1+ \frac{1}{m} + \frac{C_{1}}{\log{X}})}{e V^2(\log{X})^{2m}} \r)^{k}
+ \l( \frac{C_{1} k^{2(m+1)}}{V^2 (\log{T})^{2m}} \r)^{k}.
\end{align}
Here, if $V$ satisfies $2(\log{X})^{-m} \leq V \leq c_{0}(\log{T})^{\frac{m}{2m+1}} (\log{X})^{-\frac{2m^2 + 2m}{2m+1}}$, 
then we choose $k = [V^2 (\log{X})^{2m}/4(1+1/m)]$, 
where $c_{0}$ is an absolute positive constant satisfying $c_{0} \leq e^{-1}C_{1}^{1/(4m+2)}$.
Then, by \eqref{PJIN1}, we have
\begin{align}	\label{JIN1}
&\meas(\T_{m}(T, X, V))
\ll \exp\l( - \frac{m}{4(m+1)}V^2 (\log{X})^{2m} \l(1 - \frac{C'}{\log{X}}\r) \r).
\end{align}
If $V$ satisfies $c_{0}(\log{T})^{\frac{m}{2m+1}} (\log{X})^{-\frac{2m^2 + 2m}{2m+1}} 
\leq V \leq \frac{\log{T}}{(\log{X})^{m+1}}$, 
then we choose $k = [(e C_1)^{-\frac{1}{m+1}} V^{\frac{1}{m+1}} (\log{T})^{\frac{m}{m+1}}]$.
Then, by \eqref{PJIN1}, we have
\begin{align}	\label{JIN2}
\meas(\T_{m}(T, X, V))
\ll \exp\l(-c_{1} V^{\frac{1}{m+1}} (\log{T})^{\frac{m}{m+1}} \r).
\end{align}
Thus, from estimates \eqref{JIN1} and \eqref{JIN2}, we obtain this theorem.

Next, we show \eqref{Rmk2} under the Riemann Hypothesis.
Let $f$ be a fixed function satisfying the condition of this paper and $D(f) \geq 2$. 
By Theorem \ref{Main_Prop} as $H = 1$, for $X \leq Z \leq T$, we have
\begin{multline}	\label{KDA}
\eta_{m}(\s + it) 
- i^{m}\sum_{2 \leq n \leq X}\frac{\Lam(n)}{n^{\s + it}(\log{n})^{m + 1}}\\
= i^{m}\sum_{X < n \leq Z^2}\frac{\Lam(n)v_{f, 1}(e^{\log{n} / \log{Z}})}{n^{\s + it}(\log{n})^{m + 1}}
+ R_{m}(\s + it, Z, 1).
\end{multline}
Since we assume the Riemann Hypothesis, 
by using Proposition \ref{RCSFP}, 
it holds that there exists some constant $C_3 > 1$ such that for any $3 \leq Z \leq T$, $t \in [T, 2T]$,
\begin{align*}	
|R_{m}(1/2 + it, Z, 1)|
\leq \frac{C_{3}}{2}\l( \frac{1}{(\log{Z})^{m+1}}
\l| \sum_{p \leq Z^3}\frac{w_{Z}(p)\log{p}}{p^{\frac{1}{2} + \frac{4}{\log{Z}} + it}} \r| 
+ \frac{\log{T}}{(\log{Z})^{m+1}} \r),
\end{align*}
where $w_{Z}$ is defined by \eqref{def_w_X}.
Therefore, by letting $Z = \exp\l( \l(C_3\frac{\log{T}}{V}\r)^{\frac{1}{m+1}} \r)$, we have
\begin{align*}
|R_{m}(1/2 + it, Z; u)|
\leq \frac{V}{2\log{T}}\l| \sum_{p \leq Z^2}\frac{w_{Z}(p)\log{p}}{p^{\frac{1}{2} + \frac{4}{\log{Z}} + it}} \r| + \frac{V}{2}
\end{align*}
for $t \in [T, 2T]$. Note that the inequality $V \leq \frac{\log{T}}{(\log{X})^{m+1}}$ implies $X \leq Z$.
Hence, by formula \eqref{KDA}, when $V \leq \frac{\log{T}}{(\log{X})^{m+1}}$, we have
\begin{align}	\label{BINJS}
\meas(\T_{m}(T, X, V))
\leq \meas(S_1) + \meas(S_2).
\end{align}
Here, the sets $S_1$ and $S_2$ are defined by
\begin{gather*}
S_1 := \l\{ t \in [T, 2T] \; \Bigg{|} \; 
\bigg| \sum_{X < n \leq Z^2}\frac{\Lam(n)v_{f, 1}(e^{\log{n} / \log{Z}})}{n^{1/2 + it}(\log{n})^{m+1}} \bigg| > \frac{V}{4} \r\},\\
S_2 := \l\{ t \in [T, 2T] \; \Bigg{|} \; 
\frac{V}{2\log{T}}\bigg| \sum_{p \leq Z^3}\frac{w_{Z}(p)\log{p}}{p^{\frac{1}{2} + \frac{4}{\log{Z}} + it}} \bigg| 
> \frac{V}{4} \r\}.
\end{gather*}

By the same calculation as \eqref{QvM2}, we obtain
\begin{align}	\label{IEMBT2}
\frac{1}{T}\int_{T}^{2T}
\Bigg| \sum_{X < n \leq Z^2}\frac{\Lam(n)v_{f, 1}(e^{\log{n} / \log{Z}})}{n^{1/2+it}(\log{n})^{m+1}} \Bigg|^{2k}dt
\ll \frac{C^{k} k!}{(\log{X})^{2m k}}.
\end{align}

On the other hand, by Lemma \ref{SLL} and the prime number theorem, we find that
\begin{align}	\label{CS2}
\frac{1}{T}\int_{T}^{2T}\l( \frac{V}{2\log{T}}
\l| \sum_{p \leq Z^3}\frac{w_{Z}(p) \log{p}}{p^{\frac{1}{2} + \frac{4}{\log{Z}}+it}} \r| \r)^{2k}dt
\ll C^k k!\l(\frac{V}{\log{T}}\r)^{\frac{2m}{m+1}k}
\end{align}
for $k \leq c_0 V^{\frac{1}{m+1}} (\log{T})^{\frac{m}{m+1}}$. Here $c_0$ is a small positive constant.
Therefore, by this estimate and \eqref{IEMBT2}, we obtain the following estimates
\begin{align}	\label{ESSS}
\begin{gathered}
\frac{\meas(S_1) + \meas(S_2)}{T}
\ll \l( \frac{C_{4} k^{1/2} }{V(\log{X})^{m}} \r)^{2k}+
 \l( \frac{C_{4} k^{1/2} }{V} \l( \frac{V}{\log{T}} \r)^{m/(m+1)} \r)^{2k},
\end{gathered}
\end{align}
where $C_4$ is a sufficiently large positive constant.
Hence, by these esitmates and \eqref{BINJS}, when $V \leq \frac{\log{T}}{(\log{X})^{m+1}}$, we have
\begin{align*}	
\meas(\T_{m}(T, X, V))
\ll \l( \frac{C_{4} k^{1/2}}{V(\log{X})^{m}} \r)^{2k}.
\end{align*}
Since $V$ satisfies $(\log{T})^{\frac{m}{2m+1}} (\log{X})^{-\frac{2m^2 + 2m}{2m+1}} 
\leq V \leq \frac{C_{0}\log{T}}{(\log{X})^{m+1}}$, choosing $k = [(e C_4)^{-2} V^{\frac{1}{m+1}} (\log{T})^{\frac{m}{m+1}}]$, we have
\begin{align*}	
\meas(\T_{m}(T, X, V))
\ll \exp\l(-c_4 V^{\frac{1}{m+1}} (\log{T})^{\frac{m}{m+1}} 
\log\l( e\frac{V^{\frac{2m+1}{2m+2}}(\log{X})^{m}}{(\log{T})^{\frac{m}{2m+2}}} \r)\r).
\end{align*}
Thus, we obtain estimate \eqref{Rmk2} under the Riemann Hypothesis.
\end{proof}

\begin{acknowledgment*}
The author would like to deeply thank Mr Kenta Endo for useful discussion, 
and some results in this paper are motivated in the discussion with him.
The author would like to thank Professors Kohji Matsumoto and Hidehiko Mishou for their helpful comments.
Finally, the author would like to thank Professor Scott Kirila and Mr Masahiro Mine for telling me the information on some papers.
This work is supported by Grant-in-Aid for JSPS Research Fellow (Grant Number: 19J11223).
\end{acknowledgment*}





\end{document}